\documentclass[12pt]{article}%
\usepackage{amsmath,amssymb,amsthm,amsfonts,verbatim}
\usepackage{bm, bbm, mathrsfs}
\RequirePackage{hyperref}
\usepackage{fullpage}
\usepackage{graphicx}
\usepackage{xypic}
\usepackage{pstricks}
\usepackage{enumitem}
\usepackage{array}
\allowdisplaybreaks

\newtheorem{theorem}{Theorem}
\newtheorem{lemma}{Lemma}

\newtheorem{example}[lemma]{Example}

\newtheorem{proposition}[lemma]{Proposition}

\newtheorem{remark}[lemma]{Remark}

\newtheorem{definition}[lemma]{Definition}
\def\k{\mathbb{K}}
\newcommand{\syz}{\mathrm{syz}}
\newcommand{\rank}{\mathrm{rank}}
\allowdisplaybreaks

\begin{document}

\title{Algorithm for computing \\
$\mu$-bases of univariate polynomials}
\author{Hoon Hong, Zachary Hough, Irina A. Kogan}
\date{}
\maketitle

\abstract {We present a new algorithm for computing a $\mu$-basis of the syzygy module of $n$ polynomials in one variable over an arbitrary field $\k$. The algorithm is conceptually  different from the previously-developed algorithms  by Cox, Sederberg, Chen, Zheng, and Wang 
for $n=3$, and by Song and Goldman   
for an arbitrary $n$.
 {The algorithm} involves computing a  ``partial'' reduced row-echelon  form of a  $ (2d+1)\times n(d+1)$ matrix over $\k$, where $d$ is the maximum degree of the input  polynomials. The proof of the algorithm is  based on standard linear algebra and is completely self-contained. {The proof} includes a proof of the existence of the $\mu$-basis  and  as a consequence provides an alternative proof of the freeness of the syzygy module.
The theoretical (worst case asymptotic) computational complexity  of the algorithm is $O(d^2n+d^3+n^2)$. We have implemented this algorithm (HHK)\ and the one developed by Song and Goldman  (SG). Experiments on random inputs indicate that 
SG is faster than HHK when $d$ is sufficiently large for a fixed $n$, and that
HHK is faster than SG when $n$ is sufficiently large for a fixed $d$.
 }

%%%

\section{Introduction}

%%%%
Let $a[s]=[a_1(s),\dots, a_n(s)]$ be a vector of univariate polynomials
over a field $\k$. It is well-known that the syzygy module of $a$, consisting of linear relations over $\k[s]$ among $a_1(s),\dots, a_n(s)$:
$$\syz(a)=\{h\in  \k[s]^{n}\,|\,a_1\,h_1+\dots+a_n\,h_n=0 \}$$ 
is free.\footnote{ Freeness of the syzygy module in the one-variable can be deduced  from the Hilbert Syzygy Theorem \cite{hilbert-1890}. In the multivariable case, the syzygy module of a polynomial vector is not always free (see, for instance, \cite{cloG})} This means that the syzygy module  has a basis, and, in fact, infinitely many bases. A  $\mu $-basis is a basis 
with  particularly nice properties, which we describe in more detail in the next section. 

The concept  of a $\mu$-basis first appeared in  \cite{cox-sederberg-chen-1998}, motivated by  the search  for new, more efficient  methods for solving implicitization problems for rational curves, and as  a further development of the method of moving lines (and, more generally, moving curves) proposed in  \cite{sederberg-chen-1995}. 
Since then, a large body of literature on the applications of $\mu$-bases to various problems involving vectors  of univariate polynomials has appeared, such as \cite{chen-cox-liu-2005, song-goldman-2009, jia-goldman-2009,tesemma-wang-2014}.\footnote{A notion of a $\mu$-basis for vectors of polynomials in two variables also has been developed and applied to the study of rational surfaces in three-dimensional projective space (see, for instance, \cite{chen-cox-liu-2005,shi-wang-goldman-2012}). This paper is devoted solely to the one-variable case.} The variety of possible applications motivates the development of  algorithms for computing $\mu$-bases. Although a   proof of the existence  of a $\mu$-basis for arbitrary $n$  appeared  already in \cite{cox-sederberg-chen-1998}, the algorithms were first developed for the $n=3$ case only \cite{cox-sederberg-chen-1998, zheng-sederberg-2001, chen-wang-2002}. The first algorithm for arbitrary $n$ appeared in \cite{song-goldman-2009}, as a generalization of \cite{chen-wang-2002}. 

This paper presents an alternative algorithm for an arbitrary $n$. 
The proof   of the algorithm does not rely on previously established theorems about  the freeness of the syzygy module or the existence of a $\mu$-basis, and, therefore, as a by-product, provides an alternative, self-contained, constructive proof of these facts.
In the rest of the introduction, we informally  sketch the main idea underlying this new  algorithm, compare it with previous algorithms, and briefly describe its performance.

\vskip2mm \noindent\emph{Main idea:} 
 It is well-known that the syzygy module of $a$,  $\syz(a),$  is generated by the set 
$\syz_d(a) $ of syzygies of degree at most $d=\deg(a)$. The set $
\syz_d(a)$ is obviously a $\k$-subspace of $\k[s]^n$. Using the standard monomial basis, it is easy to see that this subspace is
isomorphic to the kernel of a certain linear map $A\colon \k%
^{n(d+1)}\to \k^{2d+1}$ (explicitly given by \eqref{eqA} below). Now
we come to the {\em key} idea: one can {\em systematically} choose a suitable finite subset of
the kernel of $A$ so that the corresponding subset of $\syz_d(a)$ forms  a $\mu$%
-basis. We elaborate on how this is done. Recall that a column of a matrix
is called \emph{non-pivotal} if it is either the first column and zero, or it is a
linear combination of the previous columns. Now we  observe and prove a remarkable fact:  the set of
indices of non-pivotal columns of $A$ splits into exactly $n-1$ sets of
modulo-$n$-equivalent integers. By taking the smallest representative in
each set, we obtain $n-1$ integers, which we call \emph{basic
non-pivotal} indices. The set of non-pivotal indices of $A$ is equal to the
set of non-pivotal indices of its  reduced row-echelon form $E$. From each
non-pivotal column of $E$, an element of $\mathrm{ker}(A)$ can easily be
read off, that, in turn, gives rise to an element of $\syz(a)$,
which we call \emph{a row-echelon syzygy}. We prove that the row-echelon
syzygies corresponding to the $n-1$ \emph{basic non-pivotal} indices
comprise a $\mu$-basis. Thus, a $\mu$-basis can be found by computing the reduced row-echelon form of a single $(2d+1)\times n(d+1)$ matrix $A$ over~$\k$.
 Actually, it is sufficient to compute only a ``partial'' reduced row-echelon form containing  only  the basic non-pivotal columns  and the preceding pivotal columns.
 %%%% 
\vskip2mm \noindent\emph{Relation to the previous algorithms:} 
Cox, Sederberg and Chen \cite{cox-sederberg-chen-1998} implicitly suggested an algorithm 
for the $n=3$ case. Later, it was explicitly described in the
Introduction of \cite{zheng-sederberg-2001}.  The algorithm relies  on the fact that, in the $n=3$ case,  there are only two  elements in a  $\mu$-basis,  and their degrees 
(denoted as $\mu_1$ and $\mu_2$) 
can be determined \emph{prior} to computing the basis (see Corollary 2 on p.~811 of 
\cite{cox-sederberg-chen-1998}  and p.~621 of \cite{zheng-sederberg-2001}). 
Once the  degrees are determined, two  syzygies are constructed  from null vectors of two linear maps
  $A_1\colon \k^{3(\mu_1+1)}\to \k^{\mu_1+d+1}$  and $A_2\colon \k^{3(\mu_2+1)}\to \k^{\mu_2+d+1}$
(similar to the one described above).
Special care is taken to ensure that these syzygies are linearly independent over $\k[s]$. These  two syzygies comprise a $\mu$-basis. 
It is not clear, however, how this method can be generalized to arbitrary $n$. First,  as far as we are aware, there is not yet an efficient way to  determine  the degrees of $\mu$-basis members  {\it a priori}. Second, there is not yet an efficient way for choosing
appropriate null vectors so that the resulting syzygies are
linearly independent.

 Zheng and Sederberg \cite{zheng-sederberg-2001} gave a different algorithm for the $n=3$ case,  based on Buchberger-type reduction.  A more efficient modification was proposed by   Chen and Wang \cite{chen-wang-2002}, and was subsequently generalized to arbitrary $n$ by  Song and Goldman \cite{song-goldman-2009}.     The general algorithm   starts by observing that the set of the obvious syzygies $\{[\quad -a_j \quad a_i\quad]\,|\, 1\leq i<j\leq
n\}$ generates $\syz(a)$, provided $\gcd(a)=1$. Then Buchberger-type
reduction is used to reduce the degree of one of the syzygies at a time. It
is proved that when such reduction becomes impossible, one is left with
exactly $n-1$ non-zero syzygies that comprise a $\mu$-basis. If $\gcd(a)$ is
non-trivial, then the output is a $\mu$-basis multiplied by $\gcd(a)$. We note that, in contrast, the algorithm
developed in this paper outputs a $\mu$-basis even in the case when $\gcd(a)$
is non-trivial. See Section~\ref{sect-discussion} for more details.

%%%
\vskip2mm \noindent\emph{Performance:} 
We show that the algorithm in this paper  has theoretical complexity $O(d^2n+d^3+n^2)$, assuming that the  arithmetic takes constant time (which is the case when the field $\k$ is finite).
  We have implemented our algorithm (HHK),\ as well as Song and Goldman's 
\cite{song-goldman-2009} algorithm (SG) in Maple~\cite{maple}. 
Experiments on random inputs indicate that 
SG is faster than HHK when $d$ is sufficiently large for a fixed $n$ and that
HHK is faster than SG when $n$ is sufficiently large for a fixed $d$.

\vskip2mm \noindent\emph{Structure of the paper:} 
In Section~\ref%
{sect-problem}, we give a rigorous definition of a $\mu$-basis, describe its characteristic properties, and formulate
the problem we are considering. In Section~\ref{sect-syzd}, we prove several
lemmas about the vector space of syzygies of degree at most $d$, and the
role they play in generating the syzygy module. In Section~\ref{sect-resyz},
we define the notion of \emph{row-echelon syzygies} and explain how they can be
computed. This section contains our main  theoretical result, Theorem~\ref{thm-mubasis}, which explicitly identifies a subset of  \emph{row-echelon syzygies} that comprise a $\mu$-basis.
In
Section~\ref{sect-alg}, we present an algorithm for computing a $\mu$-basis.
In Section~\ref{sect-compl}, we analyze the theoretical (worst case
asymptotic) computational complexity of this algorithm. 
 In Section~\ref{sect-experiment}, we discuss implementation and
experiments, and  compare  the performance of  the algorithm
presented  here with the one described in~\cite{song-goldman-2009}.
We conclude the paper with a more in-depth discussion and comparison with previous works on $\mu$-bases and related problems in Section~\ref{sect-discussion}.  

%\section{Problem}
\section{$\mu$-basis of the syzygy module.}
%%%%
\label{sect-problem}  
Throughout this paper, $\k$ denotes a
field and  $\k[s]$ denotes a ring of polynomials in one
indeterminate~$s$. 
  The symbol $n$ will be reserved for the length of the polynomial vector $a$, whose syzygy module we are considering, and  from now on we assume $n>1$, because for the $n=1$ case the problem is trivial.  The symbol $d$ is reserved for the  degree of $a$. %\footnote{In several previous papers on $\mu$-bases, $n$ was  used for the degree of $a$ and, in \cite{cox-sederberg-chen-1998}, $d$ was used to denote $\text{length}(a)-1$  (the dimension of the ambient space of the rational curve defined by $a$).}  
We also will assume that $a$ is a \emph{non-zero vector}. All vectors are implicitly assumed to be \emph{column
vectors}, unless specifically stated otherwise.  Superscript $\phantom{.}^T$ denotes transposition.

\begin{definition}[Syzygy]
Let $a=[a_{1},\dots ,a_{n}]\in \k[s]^{n}$ be a \emph{row} $n$-vector
of polynomials. The \emph{syzygy} set of $a$ is 
\begin{equation*}
\syz(a)=\{h\in \k[s]^{n}\,|\,a\,h=0\}.
\end{equation*}
\end{definition}
\noindent We emphasize that $h$ is by default a column vector and   $a$  is explicitly  defined to be a row vector, so that the product  $a\,h$ is well-defined. It is easy to check that  $\syz(a)$ is a $\k[s]$-module.  To define  a $\mu$-basis, we need  the following terminology:
%%%

%%%%%%

\begin{definition}[Leading vector]\label{def-basic}
For $h\in \k[s]^{n}$ we define the \emph{degree} and the \emph{%
leading vector} of $h$ as follows:

\begin{itemize}
\item $\displaystyle{\deg (h)=\max_{i=1,\dots ,n}\deg (h_{i})}$.

\item $LV(h)=[\mathrm{coeff}(h_{1},t),\dots ,\mathrm{coeff}(h_{n},t)]^T\in 
\k^{n}$, where $t=\deg (h)$ and $\mathrm{coeff}(h_{i},t)$ denotes
the coefficient of $s^{t}$ in $h_{i}$.
\end{itemize}
\end{definition}

\begin{example}
Let $h=\left[ 
\begin{array}{c}
1-2s-2s^{2}-s^{3} \\ 
2+2s+s^{2}+s^{3} \\ 
-3%
\end{array}%
\right] $. Then $\deg (h)=3$ and $LV(h)=\left[ 
\begin{array}{r}
-1 \\ 
1 \\ 
0%
\end{array}%
\right] .$
\end{example}
%%%%
Before giving the definition of a $\mu$-basis, we state a proposition that asserts the equivalence of several statements, each of which can be taken as a definition of a  $\mu$-basis. 
%%%%%
\begin{proposition}\label{prop-equiv} For a subset $u=\{u_1,\dots,u_{n-1}\}\subset \syz(a)$, ordered so that $\deg (u_1)\leq\dots\leq \deg (u_{n-1})$, the following properties are equivalent:
%%%%
 \begin{enumerate}
 \item  \label{pr-LV}  {\em [independence of the leading vectors]}  The set $u$ generates $\syz(a)$, and the leading vectors $LV(u_{1}),\dots ,LV(u_{n-1})$ are independent over $\k$.
\item\label{pr-min} {\em [minimality of the degrees]}  The set $u$ generates $\syz(a)$,  and  if $h_1, \dots, h_{n-1}$ is any generating set of   $\syz(a)$, such that $\deg (h_1)\leq\dots\leq \deg (h_{n-1})$, then $\deg (u_i)\leq \deg (h_i)$ for $i=1,\dots, n-1$. %
\item \label{pr-sum}{\em [sum of the degrees}]  The set $u$ generates $\syz(a)$, and $\deg (u_1)+\dots + \deg(u_{n-1})=\deg(a)-\deg(\gcd(a))$.
\item \label{pr-reduced}{\em [reduced representation]}  For every $h\in \syz(a)$, there exist $g_1,\dots, g_{n-1}\in\k[s]$ such that 
$\deg(g_i)\leq \deg(h)-\deg(u_i)$ and 
\begin{equation}\label{eq-reduced} h=\sum_{i=1}^{n-1} g_i\,u_i.
\end{equation}
\item \label{pr-outer} {\em [outer product]} There exists a non-zero constant $\alpha\in \k$ such that the outer product of 
$u_1,\dots, u_{n-1}$ is equal to $\alpha\, a/\gcd(a)$.
\end{enumerate} 

\end{proposition}
Here and below $\gcd(a)$ denotes the greatest common monic devisor of the polynomials $a_1, \dots, a_n$.
The above proposition is a slight rephrasing of  Theorem~2 in \cite{song-goldman-2009}. The only notable difference is that we drop the assumption that $\gcd(a)=1$ and modify Statements~\ref{pr-sum} and~\ref{pr-outer} accordingly.
After making  an observation that $\syz (a)=\syz\left(a/{\gcd(a)}\right)$, one can easily check that
a proof of Proposition~\ref{prop-equiv} can follow the same lines  as the proof of   Theorem~2 in \cite{song-goldman-2009}. \emph {We do not use Proposition~\ref{prop-equiv}  to derive and justify our algorithm for computing a $\mu$-basis, and therefore we are not including its proof.} 
We include  this proposition to underscore several important properties of a $\mu$-basis and to facilitate comparison with the  previous work on the subject.

Following  \cite{song-goldman-2009}, we base our definition of a $\mu$-basis on Statement~\ref{pr-LV} of Proposition~\ref{prop-equiv}. We are making this choice, because in the process of proving the existence of a $\mu$-basis, we explicitly construct a set of $n-1$ syzygies  for which  Statement~\ref{pr-LV} can be easily verified, while verification of  the other statements of Proposition~\ref{prop-equiv} is not immediate.
  The original definition of a $\mu$-basis (p.~824 of  \cite{cox-sederberg-chen-1998}) is  based on  the sum of the degrees property (Statement~\ref{pr-min} of Proposition~\ref{prop-equiv}). In Section~\ref{sect-discussion}, we  discuss the advantages of the original definition.

%%%%%%
{\begin{definition}[$\protect\mu $-basis]
\label{def-mubasis} For a non-zero row vector $a\in \k[s]^{n}$,  a subset $%
u\subset \k[s]^{n}$  of polynomial vectors is called a \emph{$\mu $-basis of $a$}, or, equivalently, a \emph{$\mu$-basis of $\syz(a)$}, if the following
three properties hold:
\begin{enumerate}
\item $u\ $has exactly $n-1\ $elements;

\item $LV(u_{1}),\dots ,LV(u_{n-1})$ are independent over $\k$;

\item $u$ is a basis of $\syz(a)$, the syzygy module  of $a$.
\end{enumerate}
\end{definition}
}
%%%

As we show in  Lemma~\ref{lem-basis} below, the $\k$-linear independence of leading vectors of any set of polynomial vectors immediately implies the $\k[s]$-linear independence of  the polynomial vectors themselves. Therefore, a set $u$  satisfying Statement~\ref{pr-LV} of Proposition~\ref{prop-equiv} is a basis of $\syz(a)$.  Thus, the apparently stronger Definition~\ref{def-mubasis} is, in fact, equivalent to Statement~\ref{pr-LV} of Proposition~\ref{prop-equiv}.

%\noindent The term \emph{$\mu$-basis of $a$}, introduced in the above definition, can be thought as  an abbreviation  of  the term \emph{$\mu$-basis of $\syz(a)$}. We will use both terms interchangeably. 

In the next two sections, through a series of lemmas culminating in Theorem~\ref{thm-mubasis}, we give a self-contained constructive proof of the existence of a $\mu$-basis. This, in turn, leads to an algorithm, presented in Section~\ref{sect-alg},  for solving  the following problem: 
\vskip2mm \noindent\textbf{Problem:} 

\begin{description}
[leftmargin=5em,style=nextline,itemsep=-.02em]

\item[$Input:$] $a\neq 0 \in\k[s]^{n}$, row vector, where $n>1$ and $\k$
is a computable field.\footnote{ 
A
field is \emph{computable} if there are algorithms for carrying out the
arithmetic ($+,-,\times,/$) operations among the field elements.}

\item[$Output:$] $M\in\k[s]^{n\times\left( n-1\right)}$, such that
the columns of $M$ form a $\mu$-basis of $a$.
\end{description}

\begin{example}[Running example]
\label{ex-prob}
We will be using the following  simple example throughout the paper to illustrate the theoretical ideas/findings    and the resulting algorithm. 

\begin{description}
[leftmargin=5em,style=nextline,itemsep=-.02em]

\item[$Input$] $a=\left[ 
\begin{array}{ccc}
1+s^{2}+s^{4} & 1+s^{3}+s^{4} & 1+s^{4}%
\end{array}
\right] \in \mathbb{Q}[s]^3$

\item[$Output$] $M=\left[ 
\begin{array}{cc}
-s & 1-2s-2s^{2}-s^{3} \\ 
1 & 2+2s+s^{2}+s^{3} \\ 
-1+s & -3%
\end{array}%
\right] $
\end{description}
\end{example}
%%%
In contrast to the algorithm developed by  Song and Goldman in \cite{song-goldman-2009},  the algorithm presented  in this paper produces  a $\mu$-basis even when the input vector $a$ has a non-trivial     greatest common divisor (see Section~\ref{sect-discussion} for more details).

 It is worthwhile emphasizing that not every basis of $\syz(a)$ is a $\mu$-basis. Indeed, let $u_1$ and $u_2$ be the columns of matrix $M$ in Example~\ref{ex-prob}. Then $u_1+u_2$ and $u_2$ is a basis of $\syz(a)$, but not a $\mu$-basis, because $LV(u_1+u_2)=LV(u_2)$.  A $\mu$-basis is not canonical: for instance,   $u_1$ and $u_1+u_2$ will provide another $\mu$-basis for  $\syz(a)$  in Example~\ref{ex-prob}.  However, Statement~\ref{pr-min} of Proposition~\ref{prop-equiv} implies that the degrees of the members of a $\mu$-basis are canonical.   In  \cite{cox-sederberg-chen-1998}, these degrees were denoted by $\mu_1,\dots, \mu_{n-1}$ and the term ``$\mu$-basis'' was coined.  
%%%
 {A more in-depth comparison with previous works on $\mu$-bases and discussion of some related problems can be found in Section~\ref{sect-discussion}.}
%We conclude this section with a more in-depth discussion and comparison with previous works on $\mu$-basis and related problems. A reader who is inclined  to get straight to the core of the paper may safely proceed to the next section. 
%%
%%%%%%%%%%

%%%%%
\section{Syzygies of bounded degree.}

\label{sect-syzd}

From now on, let $\left\langle \square \right\rangle _{\k[s]}$ stand
for the $\k[s]$-module generated by $\square$. It is known that $\syz(a)$ is
generated by polynomial vectors of degree at most $d=\deg(a)$. To keep our
presentation self-contained, we provide a proof of this fact (adapted from
Lemma 2 of \cite{song-goldman-2009}).

\begin{lemma}
\label{lem-d} Let $a\in\k[s]^n$ be of degree $d$. Then $\syz%
(a)$ is generated by polynomial vectors of degree at most $d$.
\end{lemma}

\begin{proof}
Let $\tilde{a}=a/gcd(a)=[\tilde{a}_{1},\dots ,\tilde{a}_{n}]$. For all $i<j$%
, let 
\begin{equation*}
u_{ij}=[\quad -\tilde{a}_{j}\quad \tilde{a}_{i}\quad ]^{T},
\end{equation*}%
with $-\tilde{a}_{j}$ in $i$-th position, $\tilde{a}_{i}$ in $j$-th position,
and all the other elements equal to zero. We claim that the $u_{ij}$'s are
the desired polynomial vectors. First note that 
\begin{equation*}
\max_{1\leq i<j\leq n}\deg (u_{ij})=\max_{1\leq i\leq n}\tilde{a}_{i}\leq
\deg a=d.
\end{equation*}%
It remains to show that $\syz(a)=\left\langle u_{ij}\,|\,1\leq
i<j\leq n\right\rangle _{\k[s]}.$ Obviously we have 
\begin{equation}
\syz(a)=\syz(\tilde{a})  \label{inc0}
\end{equation}%
Since $u_{ij}$ belongs to $\syz(\tilde{a})$, we have%
\begin{equation}
\syz(\tilde{a})\supset \left\langle u_{ij}\,|\,1\leq i<j\leq
n\right\rangle _{\k[s]}.  \label{inc1}
\end{equation}%
Since $\gcd (\tilde{a})=1$, there exists a polynomial vector $f=[f_{1},\dots
,f_{n}]{^{T}}$ such that 
\begin{equation*}
\tilde{a}_1f_1+\cdots+\tilde{a}_nf_n=1.
\end{equation*}%
For any $h=[h_{1},\dots ,h_{n}]^{T}\in \syz(\tilde{a})$, by definition
$$
\tilde{a}_1h_1+\cdots+\tilde{a}_nh_n = 0.
$$
Therefore, for each $h_i$,
\begin{align*}
h_i &= (\tilde{a}_1f_1+\cdots+\tilde{a}_nf_n)h_i\\
&= \tilde{a}_1f_1h_i+\cdots+\tilde{a}_{i-1}f_{i-1}h_i+\quad\tilde{a}_if_ih_i\quad+\tilde{a}_{i+1}f_{i+1}h_i+\cdots+\tilde{a}_nf_nh_i\\
&= \tilde{a}_1f_1h_i+\cdots+  \tilde{a}_{i-1}f_{i-1}\,h_i-f_i\sum_{k\neq i, k=1}^n\tilde{a}_kh_k \,+  \tilde{a}_{i+1}f_{i+1}h_i+\cdots+\tilde{a}_nf_nh_i\\
&= \tilde{a}_1(f_1h_i - f_ih_1)+\cdots+\tilde{a}_n(f_nh_i-f_ih_n)=\sum_{k\neq i, k=1}^n\,[k,i]\,\tilde{a}_k,
\end{align*}
where we denote $f_k\,h_i-f_ih_k$ by $[k,i]$.
Since $[k,i] = -[i,k]$, it follows that
$$
h = [h_1,\ldots,h_n]^T = \sum_{1\le i<j \le n} [i,j][\quad -\tilde{a}_{j}\quad \tilde{a}_{i}\quad ]^{T} .
$$
That is,
\begin{equation*}
h=\sum_{1\leq i<j\leq n}(f_{i}h_{j}\,-f_{j}h_{i})u_{ij}.
\end{equation*}
Therefore 
\begin{equation}
\syz(\tilde{a})\subset \left\langle u_{ij}\,|\,1\leq i<j\leq
n\right\rangle _{\k[s]}.  \label{inc2}
\end{equation}%
Putting \eqref{inc0}, \eqref{inc1} and \eqref{inc2} together, we have%
\begin{equation*}
\syz(a)=\left\langle u_{ij}\,|\,1\leq i<j\leq n\right\rangle _{%
\k[s]}.
\end{equation*}
\end{proof}

Let $\k[s]_d$ denote the set of polynomials of degree at most $d$,
let $\k[s]_d^n$ denote the set of polynomial vectors of degree at
most $d$, and let 
\begin{equation*}
\syz_d(a)=\{h\in \k[s]^{n}_d\,|\,a\,h=0 \}
\end{equation*}
be the set of all syzygies of degree at most $d$.

It is obvious that $\k[s]_{d}$ is a $(d+1)$-dimensional vector space
over $\k$. 
Therefore, the set $\k[s]_{d}^{n}$ is an $n\,(d+1)$-dimensional
vector space over $\k$. It is straightforward to check that $\mathrm{syz}_{d}(a)$ is a vector subspace of $\k[s]_{d}^{n}$ over $\k
$ and, therefore, is finite-dimensional. 
%\footnote{From the proof of Lemma~\ref{lem-card} below it follows that that $\dim(\syz_d(a))=(n-1)+d(n-2)$, but we will not be using that fact directly.}
The following lemma states that a $\k$-basis of the vector space $\syz_{d}(a)$ generates
the $\k[s]$-module $\syz(a)$. The proof of this lemma follows from Lemma~\ref{lem-d} in a few trivial steps and is included for the sake of completeness.   %%

\begin{lemma}
\label{lem-syzd-ker}Let $a\in \k[s]^{n}$ be of degree $d$ and $%
h_{1},\dots h_{l}$ be a basis of the $\k$-vector space $\syz%
_{d}(a)$. Then $\syz(a)=\left\langle h_{1},\dots ,h_{l}\right\rangle
_{\k[s]}$.
\end{lemma}

\begin{proof}
From Lemma~\ref{lem-d}, it follows that there exist $u_{1},\dots ,u_{r}\in 
\syz_{d}(a)$ that generate the $\k[s]$-module $\syz%
(a)$. Therefore, for any $f\in \syz(a)$, there exist $g_{1},\dots
,g_{r}\in \k[s]$, such that 
\begin{equation}
f=\sum_{i=1}^{r}g_{i}\,u_{i}.  \label{eq-fv}
\end{equation}%
Since $h_{1},\dots h_{l}$ is a basis of the $\k$-vector space $%
\syz_{d}(a)$, there exist $\alpha _{ij}\in \k$ such that 
\begin{equation}
u_{i}=\sum_{j=1}^{l}\alpha _{ij}\,h_{j}.  \label{eq-vh}
\end{equation}%
Combining \eqref{eq-fv} and \eqref{eq-vh} we get: 
\begin{equation*}
f=\sum_{i=1}^{r}g_{i}\,\sum_{j=1}^{l}\alpha
_{ij}\,h_{j}=\sum_{j=1}^{l}\left( \sum_{i=1}^{r}\alpha _{ij}g_{j}\right)
\,h_{j}.
\end{equation*}
\end{proof}

The next step is to show  that the vector space $\syz_d(a)$ is isomorphic to the kernel of a  linear map $A\colon \k^{n(d+1)}\to \k^{2d+1}$ defined as follows: for $a=\displaystyle{\sum_{0\leq j\leq d}c_{j}s^{j}} \in \k_d^n[s]$, where $c_j=[c_{1j},\dots, c_{nj}]\in \k^n$ are \emph{row} vectors,  define
\begin{equation}\label{eqA}
A=
\left[
\begin{array}
[c]{ccc}%
c_{0} &  & \\
\vdots & \ddots & \\
c_d & \vdots & c_{0}\\
& \ddots & \vdots\\
&  & c_{d}%
\end{array}
\right]  \in\k^{(2d+1)\times n(d+1)},
\end{equation}
with the blank spaces filled by zeros.

For this purpose, we define an explicit isomorphism between vector spaces $\k[s]_{t}^{m}$  and    $\k^{m(t+1)}$, where $t$ and $m$ are arbitrary natural numbers.
Any  polynomial $m$-vector $h$ of degree at most $t$ can be written as
$h=w_0+s w_1+\dots+s^{t} w_{t}$  where $w_i=[w_{1i},\dots, w_{mi}]^T\in \k^{m}$.  It is clear that the map
$$\sharp^{m}_{t}\colon \k[s]_{t}^{m}\to \k^{m(t+1)}$$
\begin{equation}\label{iso1} h\to h^{\sharp^m_t}=
\left[
\begin{array}{c}
  w_0   \\
\vdots   \\
w_{t}   
\end{array}
\right] \end{equation}  
is linear.
It is easy to check that the  inverse of this map
$$\flat^{m}_{t}\colon \k^{m(t+1)}   \to  \k[s]_{t}^{m}$$
 is given by a linear map:
\begin{equation}\label{iso2} v\to v^{\flat^m_t} =S^{m}_{t}\,  v        \end{equation}
 where   $$S^{m}_{t}=\left[
\begin{array}
[c]{ccc}%
I_{m} & sI_{m} & \cdots s^{t}I_{m}%
\end{array}
\right] \in \k[s]^{m\times m(t+1)} .
$$
 Here $I_m$ denotes the $m\times m$ identity matrix. {For the sake of notational simplicity, we will often write  
$\sharp$, $\flat$ and $S$ instead of $\sharp^m_t$, $\flat^m_t$ and~$S^m_t$ 
when the values of $m$ and $t$ are clear from the context.}
  
\begin{example}\label{ex-sharp-flat}\rm For  
   $$h=\left[\begin{array}
[c]{c}%
 1-2s-2s^{2}-s^{3}\\
 2+2s+s^{2}+s^{3}\\
 -3
\end{array}
\right]= \left[\begin{array}
[c]{r}%
 1\\
 2\\
 -3
\end{array}
\right]+s\,\left[\begin{array}
[c]{r}%
 -2\\
 2\\
 0
\end{array}
\right]
+
s^2\,\left[\begin{array}
[c]{r}%
 -2\\
 1\\
 0
\end{array}
\right]+
s^3\,\left[\begin{array}
[c]{r}%
 -1\\
1\\
0
\end{array}
\right],$$
we have 
$$h^\sharp=[1,\,2,\,-3,\,-2,\,2,\,0,\,-2,\,1,\,0,\,-1,\,1,\,0]^T.$$
Note that
$$h=(h^\sharp)^\flat=S\, h^\sharp=\left[\begin{array}{cccc} I_3 & sI_3 & s^2I_3 & s^3I_3\end{array}\right] h^\sharp.$$
  \end{example}
With respect to the isomorphisms $\sharp$ and $\flat$,  the $\k$-linear map $a\colon \k[s]_d^n\to\k[s]_{2d}$ corresponds to the $\k$ linear map $A\colon \k^{n(d+1)}\to \k^{2d+1}$ in the following sense: 
  %%%
  \begin{lemma}\label{lem-aA}
 Let $a=\displaystyle{\sum_{0\leq j\leq d}c_{j}s^{j}} \in \k_d^n[s]$ and $A \in \mathbb{K}^{(2d+1)\times n(d+1)}$ defined as in \eqref{eqA}.  Then for any $v\in \k^{n(d+1)}$:
\begin{equation}\label{eq-aA} a  v^\flat=(A v)^\flat .\end{equation}

\end{lemma}
\begin{proof}
A vector   $v\in \k^{n(d+1)}$ can be split into $(d+1)$ blocks
$$\left[
\begin{array}{c}
  w_0   \\
\vdots   \\
w_d   
\end{array}
\right],$$ 
 where $w_i\in \k^n$ are column vectors.  For $j<0$ and $j>d$, let us define $c_j=0\in \k^n$.
 Then $Av$ is a $(2d+1)$-vector with $(k+1)$-th entry  $$(Av)_{k+1}=c_k\,w_{0}+c_{k-1}w_{1}+\cdots+c_{k-d}w_{d}=\displaystyle{\sum_{0\leq i\leq d}c_{k-i}w_{i}},$$  where $k=0,\dots, 2d$. Then

\begin{eqnarray*}
a v^\flat&=&a\,S_d^n\, v=\left(  \sum_{0\leq j\leq d}c_{j}s^{j}\right)  \,\left(  \sum_{0\leq
i\leq d}w_{i}\,s^{i}\right)  =\sum_{0\leq i,j\leq d}c_{j}w_{i}s^{i+j}\\
&=&\sum_{0\leq k\leq2d}s^k \,\left(  \sum_{0\leq i\leq d}c_{k-i}w_{i}\right) = \sum_{0\leq k\leq2d}s^k\, (Av)_{k+1} =S_{2d}^1\, (Av)=(Av)^\flat.
\end{eqnarray*}

\end{proof}
\begin{example}\label{ex-A} \rm For  the row vector $a$ in  the running example  (Example~\ref{ex-prob}), we have $n=3$, $d=4$,
 $$c_{0}=[1,1,1],\, c_{1}=[0,0,0],\,c_{2}=[1,0,0],\,c_{3}=[0,1,0],\,c_{4}=[1,1,1]$$
and
$$A=\left[
\begin{array}
[c]{rrrrrrrrrrrrrrr}%
1 & 1 & 1 &    &    &     &   &    &    &    &    &    &    &   &  \\
0 & 0 & 0 & 1 & 1 & 1 &    &    &    &    &    &    &    &   &  \\
1 & 0 & 0 & 0 & 0 & 0 & 1 & 1 & 1 &    &    &    &    &   &  \\
0 & 1 & 0 & 1 & 0 & 0 & 0 & 0 & 0 & 1 & 1 & 1 &    &   &  \\
1 & 1 & 1 & 0 & 1 & 0 & 1 & 0 & 0 & 0 & 0 & 0 & 1 & 1 & 1\\
 &      &   & 1 & 1 & 1 & 0 & 1 & 0 & 1 & 0 & 0 & 0 & 0 & 0\\
 &      &   &    &    &    & 1 & 1 & 1 & 0 & 1 & 0 & 1 & 0 & 0\\
 &      &   &    &    &    &    &    &    & 1 & 1 & 1 & 0 & 1 & 0\\
 &      &   &    &    &    &     &    &    &  &     &    & 1 & 1 & 1
\end{array}
\right].  $$
Let $v=[1,2,3,4,5,6,7,8,9,10,11,12,13,14,15]^T$.
Then 
$$Av= [6,15,25,39,60,33,48,47,42]^T$$
and so
$$
(Av)^\flat = S_{2d}^1(Av) = S_{8}^1(Av) = 6 + 15s + 25s^2 + 39s^3 + 60s^4 + 33s^5 + 48s^6 + 47s^7 + 42s^8.
$$
On the other hand, since 
$$v^\flat=S_d^nv = S_4^3v = \left[\begin{array}{c} 1+4s+7s^2+10s^3+13s^4\\ 2+5s+8s^2+11s^3+14s^4\\ 3+6s+9s^2+12s^3+15s^4\end{array}\right],$$ we have
\begin{align*}
a v^\flat &= \left[
\begin{array}
[c]{ccc}%
1+s^{2}+s^{4} & 1+s^{3}+s^{4} & 1+s^{4}%
\end{array}
\right]  \left[\begin{array}{c} 1+4s+7s^2+10s^3+13s^4\\ 2+5s+8s^2+11s^3+14s^4\\ 3+6s+9s^2+12s^3+15s^4\end{array}\right]\\ 
&= 42s^8+47s^7+48s^6+33s^5+60s^4+39s^3+25s^2+15s+6.
\end{align*}
We observe that
$$a v^\flat=(Av)^\flat.$$

\end{example}

%%%
\begin{lemma} \label{lem-isom}    $v\in \mathrm{ker}(A)$ if and only if  $v^\flat\in \syz_{d}(a)$. 
\end{lemma}
\begin{proof} Follows immediately from \eqref{eq-aA}.
\end{proof}
We conclude this section by describing an explicit generating set for the syzygy module. 
\begin{lemma}\label{lem-gen-syz} Let $b_1, \dots b_l$ comprise a basis of
$\mathrm{ker}(A)$, then 
$$\syz(a)=\left<b_1^\flat, \dots,  b_l^\flat\right>_{\k[s]}.$$
\end{lemma}
\begin{proof} Lemma~\ref{lem-isom} shows that the isomorphism \eqref{iso2} between vector spaces 
$\k^{n(d+1)}$ and $\k(s)_d^n$  induces an isomorphism between their respective subspaces 
$\mathrm{ker}(A)$ and $\syz_d(a)$. Therefore,  $b_1^\flat, \dots,  b_l^\flat$ is a basis of  $\syz_d(a)$.  The conclusion then follows from Lemma~\ref{lem-syzd-ker}.
\end{proof}
%%%
%Section
\section{``Row-echelon'' generators and $\mu$-bases.}\label{sect-resyz}
%%%%
In the previous section,  we proved that any basis of $\mathrm{ker}(A)$  gives rise to a generating set of $\syz(a)$.  In this section, we show that  a particular basis of $\mathrm{ker}(A)$, which can be  ``read off''  the reduced  row-echelon  form of $A$, contains $n-1$ vectors that give rise to a $\mu$-basis of $\syz(a)$.  
In this and the following sections,  $\mathrm{quo}(i,j)$ denotes the quotient  and $\mathrm{rem}(i,j)$ denotes the remainder generated by  dividing of an  integer $i$ by an integer $j$.
%%%%
\vskip2mm
\noindent We start with the following important definition:
\begin{definition} A  column of any matrix $N$ is called \emph{pivotal} if it is either the first column and  is non-zero  or it is linearly independent of all previous columns.  The rest of the columns of $N$ are called \emph{non-pivotal}. The index of  a pivotal (non-pivotal)  column is called  a \emph{pivotal (non-pivotal)} index.
\end{definition}
From this definition, using induction,  it follows that every non-pivotal column can be written as a linear combination of the preceding \emph{pivotal  columns}.

We denote the set of pivotal indices of $A$ as $p$ and the set of its non-pivotal indices as $q$.
In the following two lemmas, we show how the specific structure of the matrix $A$ is reflected in the structure of the set of non-pivotal indices $q$.

\begin{lemma}[periodicity]
\label{periodic}If $j\in q$ then $j+kn\in q$ for
$ 0\leq k\leq\left\lfloor \frac{n(d+1)-j}{n}\right\rfloor$. Moreover,
\begin{equation}\label{eq-periodic} A_{*j}=\sum_{ r<j} \alpha_r\, A_{*r}\quad \Longrightarrow\quad  A_{*j+kn}=\sum_{ r<j} \alpha_r\, A_{*r+kn},\end{equation}
where $A_{*j}$ denotes the $j$-th column of $A$.  
\end{lemma}

\begin{proof}
To prove the statement, we need to examine the structure of the $(2d+1)\times n(d+1)$ matrix $A$:
\begin{equation}
\label{At-pic}\left[
\begin{array}
[c]{cccccccccc}%
c_{01} & \cdots & c_{0n} &  &  &  &  &  &  & \\
\vdots & \cdots & \vdots & c_{01} & \cdots & c_{0n} &  &  &  & \\
\vdots & \cdots & \vdots & \vdots & \cdots & \vdots & \ddots &  &  & \\
c_{d1} & \cdots & c_{dn} & \vdots & \cdots & \vdots & \ddots & c_{01} & \cdots
& c_{0n}\\
&  &  & c_{d1} & \cdots & c_{dn} & \ddots & \vdots & \cdots & \vdots \\
&  &  &  &  & & \ddots & \vdots & \cdots & \vdots\\
&  &  &  &  &  & & c_{d1} & \cdots & c_{dn}
\end{array}
\right] .
\end{equation}
The $j$-th column of $A$
has the first $\mathrm{quo}(j-1,n)$ and the last $(d-\mathrm{quo}(j-1,n))$
entries zero. For $1 \leq j\leq nd$ the $(n+j)$-th column is obtained by
shifting all entries of the $j$-th column down by 1 and then putting an extra
zero on the top. We consider two cases:

\begin{enumerate}
\item Integer $j=1$ is in $q$ if and only if the first column of $A$ is zero.
From the structure of $A$ it follows that any column indexed by $1+kn$ is zero
and therefore, $(1+kn)\in q$ for $\left\lfloor \frac{n(d+1)-1}{n}\right\rfloor
=d \geq k\geq0$.

\item 

Let us embed  $A$ in an infinite
matrix indexed by integers. By inspection of the structure of $A$ given by
\eqref{At-pic}, we see immediately
\begin{equation}
A_{i,r+kn}=A_{i-k,r}\label{structure_of_A}.
\end{equation}
Then, for a non pivotal  index $j>1$ and $ 0\leq k\leq\left\lfloor \frac{n(d+1)-j}{n}\right\rfloor$
we have:
\begin{align*}
&  \, A_{*j}=\sum_{ r<j} \alpha_r\, A_{*r}\\
\Longleftrightarrow &  \underset{i\in\mathbb{Z}}{\forall
}\ A_{i,j}\text{ }=\sum_{r=1}^{j-1}\alpha_{r}A_{i,r}\\
\Longleftrightarrow &   \underset{i\in\mathbb{Z}}{\forall
}\ \ A_{i-k,j}\text{ }=\sum_{r=1}^{j-1}\alpha_{r}A_{i-k,r}
\;\;\;\;\;\;\;\;\;\; \text{(by reindexing the row)}\\
\Longleftrightarrow &   \underset{i\in\mathbb{Z}}{\forall
}\ \ A_{i,j+kn}\text{ }=\sum_{r=1}^{j-1}\alpha_{r}A_{i,r+kn} \;\;\;\;\;\;\;
\text{(from \eqref{structure_of_A}) }\\
\Longrightarrow & \quad   A_{*j+kn}=\sum_{ r<j} \alpha_r\, A_{*r+kn},
\end{align*}
Therefore $(j+kn)\in q$ for $\left\lfloor \frac{n(d+1)-j}{n}\right\rfloor \geq
k\geq0$ and equation \eqref{eq-periodic} holds.
\end{enumerate}
\end{proof}
 \begin{definition}  \rm Let $q$ be the set of non-pivotal indices.  Let $q/(n)$ denote the set of equivalence classes of $q
\text{ modulo } n$. Then the set   $\tilde q=\{ \min \varrho\,| \varrho\in q/(n)\}$ will be called  the set of \emph{basic non-pivotal indices}.  
 \end{definition}
 %%%

 %%%%
 \begin{example}\rm \label{ex-period} For the matrix  $ A$ in  Example~\ref{ex-A}, we have $n=3$ and  $q=\{ 6,9,11,12,14,15\}$. 
 Then $q/(n)=\big\{ \{6,9,12, 15\},\,\{11,14\} \}\big\}$ and $\tilde q=\{6,11\}$.  
 \end{example}
%%%
\begin{lemma}\label{lem-card} There are exactly  $n-1$ basic non-pivotal indices: $|\tilde q|=n-1$.
\end{lemma}
\begin{proof} We  prove this lemma by showing that   $|\tilde q|<n$ and $|\tilde q|>n-2$.\begin{enumerate}
\item Since there are at most $n$ equivalence classes in $q$ modulo $n$, it follows  from the definition of $\tilde q$ that $|\tilde q|\leq n$. Moreover,  the $(2d+1)$-th row of the last block of $n$-columns of $A$  is given by the row vector $c_d=(c_{1d}, \dots, c_{nd})=LV(a)$, which is non-zero. Thus, there exists $r\in \{1,\dots,n\}$, such that $c_{rd} \not= 0$. Suppose $r$ is minimal such that $c_{rd} \not= 0$. Then the $(nd+r)$-th column of $A$ is independent from the first $nd+r-1$ columns (since each of these columns has a zero in the $(2d+1)$-th position). Hence, there exists $r\in \{1,\dots,n\}$ such that $nd+r$ is a pivotal index. From the periodicity Lemma~\ref{periodic}, it follows that   for every  $k=0, \dots d$, index $r+kn$ is pivotal and therefore no integer  from the class $r\, \text{modulo} \,n$ belongs to $\tilde q$.   Thus $|\tilde q|< n.$
\item Assume $|\tilde q|\leq n-2$. From the periodicity Lemma~\ref{periodic}, it follows that  the set of non-pivotal  indices is the union of the sets $\{j+kn |j\in \tilde q,\, 0\leq k\leq l_j\}$, where $l_j\leq d$ is some integer. Therefore 
$$\label{ineq}| q|\leq |\tilde q|\, (d+1)\leq (n-2)(d+1)= nd +n-2d-2.$$

On the other hand, $|q|= n(d+1)-|p|$.  It is well-known (and easy to check) that $|p|=\rank (A)$. Since $\rank (A)$ cannot exceed the number of rows of $A$, $|p|\leq2d+1$. Therefore   
$$|q|\geq  n (d+1) -(2d+1)= nd+n-2d -1.$$ Contradiction. Hence $|\tilde q|>n-2$.
\end{enumerate}
\vskip3mm

\end{proof}
From the matrix $A$ we will now construct a square $n\,(d+1)\times n(d+1)$ matrix $V$ in the following way.   For $i\in p$, the $i$-th column of $V$ has   $1$ in the $i$-th row and  $0$'s in all other rows.  For $i\in q$ we define the $i$-th column from the linear relationship 
 \begin{equation}\label{eqAj}A_{*i}=\sum_{\{j\in p\,|\,  j<i\}} \alpha_j  A_{*j}\end{equation}
as follows: for $j\in p$ such that $j<i$ we set $V_{ji}=\alpha_j$. All the remaining elements of the column $ V_{*i}$ are zero.  For simplicity we will denote the $i$-th column of $V$ as $v_i$.
We note  two important properties of $V$:
\begin{enumerate}
\item Matrix $V$ has the same  linear relationships  among its columns as $A$.
\item Vectors $\{b_i=e_i -v_i \,|\, i\in q\}$, where by $e_i$ we  denote a column vector  that has $1$ in the $i$-th  position and  $0$'s in all others, comprise a basis of $\mathrm{ker}(A)$.
\end{enumerate}
The corresponding syzygies $\{b_i^\flat \,|\, i\in q\}$ will be called \emph{row-echelon syzygies}  because the  $\alpha$'s appearing in \eqref{eqAj} can be read off the reduced row-echelon  form $E$ of $A$.   (We remind the reader that the   $(2d+1)\times n(d+1)$ matrix $E$ has the following property:   for all  $i\in q$, the non-zero entries of the $i$-th  column  consist of
${\{\alpha_j\,|\, j\, \in p,\,  j<i\}}$ and $\alpha_j$ is located in the row that corresponds to the place of $j$ in the  \emph{increasingly ordered} list $p$.) The reduced row-echelon form can be computed using Gauss-Jordan elimination or  some other methods.
 \begin{example} \rm \label{ex-resyz} For the matrix $A$ in Example~\ref{ex-A}, we have $n=3$, $d = 4$, and
$$
V = \left[\begin{array}{rrrrrrrrrrrrrrr}
1 & 0 & 0 & 0 & 0 & 0 & 0 & 0 & 0 & 0 & -1 & 0 & 0 & 1 & 1\\
0 & 1 & 0 & 0 & 0 & -1 & 0 & 0 & -1 & 0 & -2 & -1 & 0 & -1 & 1\\
0 & 0 & 1 & 0 & 0 & 1 & 0 & 0 & 1 & 0 & 3 & 1 & 0 & 0 & -2\\
0 & 0 & 0 & 1 & 0 & 1 & 0 & 0 & 1 & 0 & 2 & 1 & 0 & 0 & -1\\
0 & 0 & 0 & 0 & 1 & 0 & 0 & 0 & -1 & 0 & -2 & -1 & 0 & 0 & 1\\
0 & 0 & 0 & 0 & 0 & 0 & 0 & 0 & 0 & 0 & 0 & 0 & 0 & 0 & 0\\
0 & 0 & 0 & 0 & 0 & 0 & 1 & 0 & 1 & 0 & 2 & 1 & 0 & 0 & -1\\
0 & 0 & 0 & 0 & 0 & 0 & 0 & 1 & 0 & 0 & -1 & -1 & 0 & -1 & 0\\
0 & 0 & 0 & 0 & 0 & 0 & 0 & 0 & 0 & 0 & 0 & 0 & 0 & 0 & 0\\
0 & 0 & 0 & 0 & 0 & 0 & 0 & 0 & 0 & 1 & 1 & 1 & 0 & 1 & 0\\
0 & 0 & 0 & 0 & 0 & 0 & 0 & 0 & 0 & 0 & 0 & 0 & 0 & 0 & 0\\
0 & 0 & 0 & 0 & 0 & 0 & 0 & 0 & 0 & 0 & 0 & 0 & 0 & 0 & 0\\
0 & 0 & 0 & 0 & 0 & 0 & 0 & 0 & 0 & 0 & 0 & 0 & 1 & 1 & 1\\
0 & 0 & 0 & 0 & 0 & 0 & 0 & 0 & 0 & 0 & 0 & 0 & 0 & 0 & 0\\
0 & 0 & 0 & 0 & 0 & 0 & 0 & 0 & 0 & 0 & 0 & 0 & 0 & 0 & 0
\end{array}\right].
$$ 
 The non-pivotal indices are $q=\{ 6,9,11,12,14,15\}$. We have
\begin{align*}
b_6 &= e_6 - v_6 = [0,1,-1,-1,0,1,0,0,0,0,0,0,0,0,0]^T\\
b_9 &= e_9 - v_9 = [0,1,-1,-1,1,0,-1,0,1,0,0,0,0,0,0]^T\\
b_{11} & = e_{11} - v_{11} = [1,2,-3,-2,2,0,-2,1,0,-1,1,0,0,0,0]^T\\
b_{12} &= e_{12} - v_{12} = [0,1,-1,-1,1,0,-1,1,0,-1,0,1,0,0,0]^T\\
b_{14} &= e_{14} - v_{14} = [-1,1,0,0,0,0,0,1,0,-1,0,0,-1,1,0]^T\\
b_{15} &= e_{15} - v_{15} = [-1,-1,2,1,-1,0,1,0,0,0,0,0,-1,0,1]^T
\end{align*}
and the corresponding row-echelon syzygies are
\begin{align*}
b_6^\flat &= \left[\begin{array}{c} -s \\ 1\\ -1+s \end{array} \right] && b_9^\flat = \left[\begin{array}{c} -s-s^2 \\ 1+s \\ -1+s^2 \end{array}\right]\\
b_{11}^\flat &= \left[\begin{array}{c} 1-2s-2s^2-s^3 \\ 2+2s+s^2+s^3 \\ -3 \end{array}\right] && b_{12}^\flat = \left[\begin{array}{c} -s-s^2-s^3 \\ 1+s+s^2\\ -1+s^3 \end{array}\right]\\
b_{14}^\flat &= \left[\begin{array}{c} -1-s^3-s^4 \\ 1+s^2+s^4\\ 0 \end{array}\right]&& b_{15}^\flat = \left[\begin{array}{c} -1+s+s^2-s^4 \\ -1-s\\ 2+s^4 \end{array}\right].
\end{align*}
\end{example}

 \noindent The following lemma shows a crucial relationship between the {row-echelon syzygies}. {Note that, in this lemma, we use  $i$ to denote a non-pivotal index and $\iota$ to denote a \emph{basic} non-pivotal index.} 
%%%%
 \begin{lemma} Let $v_r$, $\, r\in\{1,\dots, n(d+1)\}$ denote columns of the matrix $V$. For $i\in q$, let 
 \begin{equation}\label{eq-b}b_i=e_i-v_i\end{equation}  
 Then for  any $\iota\in \tilde q$ and any integer $k$ such that  $ 0\leq k\leq\left\lfloor \frac{n(d+1)-\iota}{n}\right\rfloor$
\begin{equation}\label{eq-bikn}b_{\iota+kn}^\flat =  s^k\,b_{\iota}^\flat +\sum_{\{j\in p\,|\,j<\iota,\, {j+kn}\in q \}}\,\alpha_j\, b_{j+kn}^\flat,\end{equation}
 where constants $\alpha_j$ appear in the expression of the $\iota$-th column of $A$ as a linear combination of the previous pivotal columns:
$$A_{*\iota}=\sum_{\{j\in p\,|\,  j<\iota\}} \alpha_j  A_{*j}.$$
 \end{lemma}
 \begin{proof}
 We start by stating  identities, which we use in the proof.  By definition of $V$, we have for any  $j\in p$:
 \begin{equation}\label{eqVj} v_{j}= e_j\end{equation}
 and for any $\iota\in \tilde q$:
\begin{equation}\label{eqVi} v_{\iota}=\sum_{\{j\in p\,|\,  j<\iota\}} \alpha_j  v_{j}=\sum_{\{j\in p\,|\,  j<\iota\}} \alpha_j  e_{j}.\end{equation}
Since $V$ has the same linear relationships  among its columns as $A$, it inherits periodicity property \eqref{eq-periodic}. Therefore,  for  any $\iota\in \tilde q$ and any integer $k$ such that  $ 0\leq k\leq\left\lfloor \frac{n(d+1)-\iota}{n}\right\rfloor$:
\begin{equation}\label{eqVikn} v_{\iota+kn}=\sum_{\{j\in p\,|\,  j<\iota\}}  \alpha_j\, v_{j+kn}.\end{equation}
 We also will use an obvious relationship for any $r \in\{1,\dots, n(d+1)\}$ and  $ 0\leq k\leq\left\lfloor \frac{n(d+1)-r}{n}\right\rfloor$:
\begin{equation}\label{ejkn}e_{r+kn}^\flat=s^k e_r^\flat\end{equation}
and the fact that the set $\{1,\dots, n(d+1)\}$ is a disjoint union of the sets $p$ and $q$.
Then
\begin{align*}b_{\iota+kn}^\flat &= \left(e_{\iota+kn} -v_{\iota+kn} \right)^\flat=s^k e_\iota^\flat-\sum_{\{j\in p\,|\,  j<\iota\}}  \alpha_j\, v_{j+kn}^\flat &\text{by \eqref{eq-b}, \eqref{ejkn} and \eqref{eqVikn}}\\
&=s^k e_\iota^\flat- \sum_{\{j\in p\,|\,j<\iota,\, {j+kn}\in p \}}\,  \alpha_j\,v_{j+kn}^\flat-\sum_{\{j\in p\,|\,j<\iota,\, {j+kn}\in q \}}\, \alpha_j\, v_{j+kn}^\flat &\text{(disjoint union)}\\
&=s^k e_\iota^\flat- \sum_{\{j\in p\,|\,j<\iota,\, {j+kn}\in p \}}\,  \alpha_j\,e_{j+kn}^\flat-\sum_{\{j\in p\,|\,j<\iota,\, {j+kn}\in q \}}\, \alpha_j\, v_{j+kn}^\flat &\text{by \eqref{eqVj}}\\
&=s^k e_\iota^\flat- \sum_{\{j\in p\,|\,j<\iota \}}\,\alpha_j e_{j+kn}^\flat+\sum_{\{j\in p\,|\,j<\iota,\, {j+kn}\in q \}}\, \alpha_j\left( e_{j+kn}^\flat-v_{j+kn}^\flat\right) &\text{(disjoint union)}\\\
&=s^k e_\iota^\flat- \sum_{\{j\in p\,|\,j<\iota \}}\, s^k\, \alpha_j\, e_{j}^\flat +\sum_{\{j\in p\,|\,j<\iota,\, {j+kn}\in q \}}\,  \alpha_j\,b_{j+kn}^\flat  &\text{by \eqref{ejkn} and \eqref{eq-b}}\\
&=s^k \left(e_\iota- \sum_{\{j\in p\,|\,j<\iota \}}\, \, \alpha_j\, e_{j}\right)^\flat +\sum_{\{j\in p\,|\,j<\iota,\, {j+kn}\in q \}}\,  \alpha_j\,b_{j+kn}^\flat\\
&=s^k b_\iota^\flat +\sum_{\{j\in p\,|\,j<\iota,\, {j+kn}\in q \}}\,  \alpha_j\,b_{j+kn}^\flat.  &\text{\eqref{eqVi} and  \eqref{eq-b}} \end{align*}
 \end{proof}

  \begin{example}\rm\label{ex-recc} 
Continuing with Example~\ref{ex-resyz}, where $q=\{ 6,9,11,12,14,15\}$ and  $\tilde q=\{6,11\}$ and $p=\{1,2,3,4,5,7,8, 10,13\}$,
we have:
  \begin{eqnarray}
  \nonumber b_9^\flat&=& s\, b_6^\flat +1\, b_{6}^\flat,\\
   \label{eq-ex-rel}b_{12}^\flat&=&  s^2\, b_6^\flat+1\, b_{9}^\flat+0\,b_{11}^\flat,\\
    \nonumber  b_{14}^\flat&=&  s\, b_{11}^\flat+3\, b_{6}^\flat+(-1)\, b_{11}^\flat\\
    \nonumber b_{15}^\flat&=&  s^3\, b_6^\flat+(-1)\, b_{11}^\flat+1\, b_{12}^\flat+0\, b_{14}^\flat.
  \end{eqnarray}
  \end{example}
  %%
  %%%
In the next lemma, we show that the  subset of row-echelon syzygies indexed by the $n-1$ basic non-pivotal indices  is sufficient to generate $\syz(a)$. 
\begin{lemma}\label{lem-generating} Let $\tilde q$ denote the set of basic non-pivotal indices of $A$. Then
$$\syz(a)=\left<b_{r}^\flat\, |\,r \in \tilde q\right>_{\k[s]}.$$ 
 
\end{lemma}

\begin{proof} 
 Since  $\{b_i\,|\, i\in q\}$  comprise a basis of $\mathrm{ker}(A)$, we know from Lemma~\ref{lem-gen-syz} that $\syz(a)=\left<b^\flat_i\,|\, i\in q\right>_{\k[s]}.$
Equation \eqref{eq-bikn} implies that for all $i\in  q$, there exist constant $\beta$'s such that
\begin{equation} b_i^\flat =  s^k\,b_{ \iota}^\flat +\sum_{\{r\in q\,|\,r< i \}}\,\beta_r\, b_{r}^\flat,\end{equation}
where $ \iota\in \tilde q$ is equal to $i$ modulo $n$. 
It follows that inductively we can express  $b_i^\flat$ as a $\k[s]$-linear combination of  $\{b_r | r\in\tilde q\}$ and the conclusion of the lemma follows. 
\end{proof}
%%%
\begin{example}\rm \label{ex-relationship}
   Continuing with Example~\ref{ex-resyz},
we have from \eqref{eq-ex-rel}:
  \begin{eqnarray*}
  b_9^\flat&=& (s+1)\, b_6^\flat ,\\
   b_{12}^\flat&=&  (s^2+s+1)\, b_6^\flat+0\, b_{11}^\flat,\\
     b_{14}^\flat&=&  3\, b_6^\flat+(s-1)\, b_{11}^\flat,\\
     b_{15}^\flat&=&  (s^3+s^2+s+1)\, b_6^\flat+(-1)\, b_{11}^\flat.\\
  \end{eqnarray*}
  \end{example}
We next establish linear independence of the corresponding  leading  vectors: 
\begin{lemma}
\label{lem-leading} The leading vectors $LV(b_r^\flat),\,\, r\in \tilde q$ are linearly independent over $\k$.
\end{lemma}
\begin{proof}
The leading vector $LV(b_r^\flat)$ is equal to the last non-zero $n$-block in the $n(d+1)$-vector $b_r$. By construction, the last non-zero element of $b_r$ is equal to $1$ and occurs in the $r$-th position. Then $LV(b_r^\flat)$ has $1$ in $\bar r=(r\mod n)$ (the reminder of division of $r$ by $n$) position. All elements of $LV(b_r^\flat)$   positioned after $\bar r$ are zero. Since all  integers in $\tilde q$ are distinct ($\text{modulo } n$), $LV(b_r^\flat),\,\, r\in \tilde q$ are linearly independent over $\k$.

\end{proof}
\begin{example}\rm  The basic non-pivotal columns  of  the matrix   $V$ in Example ~\eqref{ex-resyz} are  columns $6$ and $11$. We previously computed 
\begin{align*}
b_6 &= e_6 - v_6 = [0,1,-1,-1,0,1,0,0,0,0,0,0,0,0,0]^T\\
b_{11} & = e_{11} - v_{11} = [1,2,-3,-2,2,0,-2,1,0,-1,1,0,0,0,0]^T.
\end{align*}
The last non-zero $n$-blocks of $b_6$ and $b_{11}$ are $[-1,0,1]$ and $[-1,1,0]$, respectively. These blocks coincide with $LV(b_6^\flat)$ and $LV(b_{11}^\flat)$ computed in  Example~\ref{ex-resyz}. We observe that these vectors are linearly independent, as expected.
\end{example}

\begin{lemma}
\label{lem-basis} Let polynomial vectors $h_1,\dots, h_{l}\in\k[s]^n$
be such that $LV(h_1),\dots,LV( h_{l})$ are independent over $\k$.
Then $h_1,\dots, h_{l}$ are independent over $\k[s]$.
\end{lemma}

\begin{proof}
Assume that $h_{1},\dots ,h_{l}$ are linearly \emph{dependent} over $\mathbb{%
K}[s]$, i.e.~there exist polynomials $g_{1},\dots ,g_{l}\in \k[s]$,
not all zero, such that 
\begin{equation}
\sum_{i=1}^{l}g_{i}\,h_{i}=0.  \label{eq-zero}
\end{equation}%
Let $m=\displaystyle{\max_{i=1,\dots ,l}}\left( \deg (g_{i})+\deg
(h_{i})\right) $ and let $\mathcal{I}$ be the set of indices on which this
maximum is achieved. Then \eqref{eq-zero} implies 
\begin{equation*}
\sum_{i\in \mathcal{I}}LC(g_{i})\,LV(h_{i})=0,
\end{equation*}%
where $LC(g_{i})$ is the leading coefficient of $g_{i}$ and is non-zero for $%
i\in \mathcal{I}$. This identity contradicts our assumption that $LV(h_{1}),\dots
,LV(h_{l})$ are linearly independent over $\k$.
\end{proof}

\begin{theorem}[Main]
\label{thm-mubasis} The set $u=\{b_{r}^{\flat }\,|\,r\in \tilde{q}\}$ is a $%
\mu $-basis of $a$.
\end{theorem}

\begin{proof}
We will check  that $u$
satisfies the three conditions of a $\mu $-basis in Definition \ref{def-mubasis}.

\begin{enumerate}
\item 
From Lemma \ref{lem-card}, there are exactly $n-1$ elements in $\tilde{q}$.
 Since $b_{r_{1}}^{ }\neq
b_{r_{2}}^{ }\ \,$for $r_{1}\neq r_{2}\in $ $\tilde{q}$ and since $\flat$ is an isomorphism, the set $u$ contains 
exactly $n-1$ elements.

\item
From Lemma~\ref{lem-leading}, we know that  the leading vectors $LV(b_r^\flat),\,\, r\in \tilde q$ are linearly independent over $\k$.

\item 
Lemma \ref{lem-generating} asserts  that the set $u$ generates $\syz%
(a)$.  By combining Lemmas \ref{lem-leading} and \ref{lem-basis}, we see that
the elements of this set are independent over $\k[s]$. Therefore   $u$ is a basis of $\syz(a)$.
\end{enumerate}
\end{proof}

\begin{remark}
\textrm{We note that by construction the last non-zero entry of vector $%
b_{r} $ is in the $r$-th position, and therefore 
\begin{equation*}
\deg (b_{r}^{\flat })=\big\lceil{r/n}\big\rceil-1.
\end{equation*}%
Thus we can determine the degrees of the $\mu $-basis elements \emph{prior
to computing the $\mu $-basis} from the set of basic non-pivotal
indices. }
\end{remark}

\begin{example}\rm 
\textrm{For the row vector  $a$ given in the running example 
(Example~\ref{ex-prob}), we determined that $\tilde q=\{6,11\}$. Therefore, prior to
computing a $\mu$-basis, we can determine the degrees of its members: $\mu_1=%
\big\lceil{\ 6/ 3}\big\rceil-1=1$ and $\mu_2=\big\lceil{\ 11/ 3}\big\rceil%
-1=3$. We now can apply Theorem~\ref{thm-mubasis} and the computation we
performed in Example~\ref{ex-resyz} to write down a $\mu$-basis: 
\begin{equation*}
b_6^\flat=\left[ 
\begin{array}{c}
-s \\ 
1 \\ 
-1+s%
\end{array}
\right] \text { and } b_{11}^\flat=\left[ 
\begin{array}{c}
1-2s-2s^{2}-s^{3} \\ 
2+2s+s^{2}+s^{3} \\ 
-3%
\end{array}
\right].
\end{equation*}
We observe that our degree prediction is correct. }
\end{example}

%%%%%%%

\section{Algorithm}
\label{sect-alg}

In this section, we describe an algorithm for computing  $\mu$-bases of univariate polynomials.
We assume that the reader is familiar with  Gauss-Jordan elimination 
(for computing reduced row-echelon forms and in turn null vectors), which can be found in any standard linear algebra textbook. 
 The theory developed in the previous sections can be recast into the following computational steps:  
\begin{enumerate}
\item {\em  Construct a matrix $A \in \k^{(2d+1)\times n(d+1)} $ whose null space
corresponds to $\syz_d(a)$.}
\item {\em Compute the reduced row-echelon form $E$ of $A$.}
\item {\em Construct a matrix $M\in\k[s]^{n\times\left(n-1\right)}$ 
whose columns form a $\mu$-basis of $a$}, as follows:

\begin{enumerate}
\item Construct the matrix $B \in \mathbb{K}^{n(d+1)\times (n-1)}$ 
whose columns are the null vectors of~$E$ corresponding to its  basic non-pivot columns: 
\begin{itemize}
\item $B_{\tilde{q}_{j},j}=1$
\item $B_{p_{r},j}=-E_{r,\tilde{q}_{j}}$ for all $r$
\item All other entries are zero
\end{itemize}
where $p$ is the list of the pivotal indices and 
 $\tilde{q}$ is the list of the basic non-pivotal indices of   $E$.

\item Translate the columns of $B$ into polynomials.
\end{enumerate}
\end{enumerate}
However,  steps 2 and 3  do some wasteful operations and they can be  improved, as follows:
\begin{itemize}
\item Note that step 2  constructs the  entire reduced row-echelon form of $A,$ even though  we only need  $n-1$ null vectors corresponding to its basic non-pivot columns.   Hence, it is natural to optimize this step  so that  only the $n-1$ null vectors
are constructed: instead of using Gauss-Jordan elimination  to compute the entire reduced row-echelon form, one  performs operations column by column only  on the pivot columns and basic non-pivot columns. One  aborts the elimination process as soon as $n-1$ basic non-pivot columns are found, resulting in  a  partial reduced row-echelon form of $A$. 

\item Note that step 3 constructs the entire matrix $B$ even though many entries are zero. Hence, it is natural to optimize this step   so that   we bypass constructing the matrix $B$, but  instead construct the matrix $M$   directly  from the matrix $E$. This is possible  because the matrix $E$  contains all the information about the  matrix $B$.

\end{itemize}
Below, we describe the resulting algorithm in more detail and illustrate its operation on our running example (Example~\ref{ex-prob}). 
\vskip5mm

\noindent {\bf  $\mu$-Basis Algorithm}
\begin{description}
[leftmargin=5em,style=nextline,itemsep=-.02em]

\item[$Input$] $a\neq 0 \in\k[s]^{n}$, row vector, where $n>1$ and $\k$
is a computable field
\item[$Output$] $M\in\k[s]^{n\times\left(  n-1\right)  }$ such that
its columns form a $\mu$-basis of $a$
\end{description}

\begin{enumerate}

\item \ \emph{Construct a matrix $A \in \k^{(2d+1)\times n(d+1)} $ whose null space
corresponds to $\syz_d(a)$.}
   \begin{enumerate} 
   \item $d\longleftarrow\deg(a)$

\item Identify the row vectors $c_{0},\ldots,c_{d} \in \k^n$ such that $a=c_{0}+c_{1}%
s+\cdots+c_{d}s^{d}$.

\item $A\longleftarrow \left[
\begin{array}
[c]{ccc}%
c_{0} &  & \\
\vdots & \ddots & \\
c_d & \vdots & c_{0}\\
& \ddots & \vdots\\
&  & c_{d}%
\end{array}
\right]$
\end{enumerate}

\item \ \emph{Construct the ``partial'' reduced row-echelon form  $E$ of $A$.}
\item[] This can be done by using Gauss-Jordan elimination (forward elimination, backward elimination, and normalization), with the following optimizations: 
\begin{itemize}
\item Stop the forward elimination   as soon as  $n-1$ basic non-pivot columns are detected.
\item Skip over periodic non-pivot columns. 
\item Carry out the row operations only on the required columns.
\end{itemize}

\item \ \emph{Construct a matrix $M\in\k[s]^{n\times\left(  n-1\right)  } $ whose columns form a $\mu$-basis of $a$.}

Let $p$ be the list of the pivotal indices and 
let $\tilde{q}$ be the list of the basic non-pivotal indices of   $E$.
\begin{enumerate}
\item Initialize an $n\times n-1 $ matrix $M$ with  $0$ in  every entry.

\item For $j=1,\ldots,n-1$

\qquad$r\leftarrow\operatorname*{rem}\left(  \tilde{q}_{j}-1,n\right)  +1$

\qquad$k\leftarrow\operatorname*{quo}\left(  \tilde{q}_{j}-1,n\right)  $

\qquad$M_{r,j\ }\leftarrow M_{r,j}+s^{k}$

\item For $i=1,\ldots,|p|$

\qquad$r\leftarrow\operatorname*{rem}\left(  p_{i}-1,n\right)  +1$

\qquad$k\leftarrow\operatorname*{quo}\left(  p_{i}-1,n\right)  $

\qquad For $j=1,\ldots,n-1$

\qquad$\qquad M_{r,j}\leftarrow M_{r,j}-E_{i,\tilde{q}_{j}}s^{k}$

\end{enumerate}
\end{enumerate}
\vskip2mm
%%%%%
%
\begin{theorem}
Let $M$ be the output of the $\mu$-Basis Algorithm on the input $a \in \k[s]^n$. Then the columns of $M$ form a $\mu$-basis for $a$.
\end{theorem}

\begin{proof}
In step 1, we construct the matrix $A$ whose null space corresponds to $\syz_d(a)$ as has been shown in Lemma \ref{lem-isom}. In step 2, we perform partial Gauss-Jordan operations on $A$ to identify the $n-1$ basic non-pivot columns of its reduced row-echelon form~$E$. In Lemma \ref{lem-card}, we showed that there are exactly  $n-1$ such columns. In step 3, we convert the basic non-pivot columns of $E$ into polynomial vectors, using the $\flat$-isomorphism described in Section~\ref{sect-syzd}, and return these polynomial vectors as  columns of the matrix $M$.  From Theorem~\ref{thm-mubasis} it follows that the columns of $M$ indeed form a $\mu$-basis of $a$, because  they satisfy the generating, leading vector, and linear independence conditions of Definition \ref{def-mubasis} of a $\mu$-basis.
\end{proof}

%%%%%%
\begin{example}
We trace the algorithm (with partial Gauss-Jordan) on the input vector from  Example~\ref{ex-prob}:
\[
a=\left[
\begin{array}
[c]{ccc}%
1+s^{2}+s^{4} & 1+s^{3}+s^{4} & 1+s^{4}%
\end{array}
\right]
\in \mathbb{Q}[s]^3
\]

\begin{enumerate}\rm
\newcolumntype{C}{>{\small\raggedleft\color{black}\arraybackslash$}p{1.2em}<{$}}
\newcolumntype{G}{>{\small\raggedleft\color{gray}\arraybackslash$}p{1.2em}<{$}}
\newcolumntype{B}{>{\small\raggedleft\color{blue}\arraybackslash$}p{1.2em}<{$}}
\newcolumntype{R}{>{\small\raggedleft\color{red}\arraybackslash$}p{1.2em}<{$}}
\newcolumntype{T}{>{\small\raggedleft\color{brown}\arraybackslash$}p{1.2em}<{$}}

\item \ \emph{Construct a matrix $A \in \k^{(2d+1)\times n(d+1)} $ whose null space
corresponds to $\syz_d(a)$}:
\begin{enumerate}
\arraycolsep=0.2em
\def\arraystretch{0.5}
\item  $d\longleftarrow4$

\item $c_{0},c_{1},c_{2},c_{3},c_{4}\longleftarrow\left[
\begin{array}
[c]{ccc}%
1 & 1 & 1
\end{array}
\right]  ,\left[
\begin{array}
[c]{ccc}%
0 & 0 & 0
\end{array}
\right]  ,\left[
\begin{array}
[c]{ccc}%
1 & 0 & 0
\end{array}
\right]  ,\left[
\begin{array}
[c]{ccc}%
0 & 1 & 0
\end{array}
\right]  ,\left[
\begin{array}
[c]{ccc}%
1 & 1 & 1
\end{array}
\right]  $

\item $A\longleftarrow\left[
\begin{array}{CCC|CCC|CCC|CCC|CCC}%
1 & 1 & 1 &  &  &  &  &  &  &  &  &  &  &  & \\
0 & 0 & 0 & 1 & 1 & 1 &  &  &  &  &  &  &  &  & \\
1 & 0 & 0 & 0 & 0 & 0 & 1 & 1 & 1 &  &  &  &  &  & \\
0 & 1 & 0 & 1 & 0 & 0 & 0 & 0 & 0 & 1 & 1 & 1 &  &  & \\
1 & 1 & 1 & 0 & 1 & 0 & 1 & 0 & 0 & 0 & 0 & 0 & 1 & 1 & 1\\
 &  &  & 1 & 1 & 1 & 0 & 1 & 0 & 1 & 0 & 0 & 0 & 0 & 0\\
 &  &  &  &  &  & 1 & 1 & 1 & 0 & 1 & 0 & 1 & 0 & 0\\
 &  &  &  &  &  &  &  &  & 1 & 1 & 1 & 0 & 1 & 0\\
 &  &  &  &  &  &  &  &  &  &  &  & 1 & 1 & 1
\end{array}
\right]$
\end{enumerate}
\item[] A blank indicates that the entry is zero due to structural reasons. 

\item \ \emph{Construct the ``partial'' reduced row-echelon form  $E$ of $A$}:
\item[] For this step, we will maintain/update the following data structures.
\begin{itemize}
\item $E$: the matrix initialized with $A$ and updated by the Gauss-Jordan process.
\item $p$: the set of the pivotal indices found.
\item $\tilde{q}$: the set of the basic non-pivotal indices found.
\item $O$: the list of the row operations, represented as follows. 
  \begin{itemize}
  \item[] $(i,i')$\ \ \ \ : swap $E_{i,j}$ with $E_{i',j}$ 
  \item[] $(i,w,i')$\,: $E_{i,j} \longleftarrow E_{i,j} + w \cdot E_{i',j}$
  \end{itemize}
  where $j$ is the current column index.
\end{itemize}
We will also indicate the update status of the columns of $E$  using  the following color codings.

\begin{tabular}{p{1em}lcl}
&\textcolor{gray}{gray}  &:& not yet updated \\
&\textcolor{blue}{blue}  &:& pivot \\
&\textcolor{red}{red}    &:& basic non-pivot  \\
&\textcolor{brown}{brown}&:& periodic non-pivot 
\end{tabular}

Now we show the trace.
\begin{enumerate}
\arraycolsep=0.2em
\def\arraystretch{0.5}
\item Initialize. \\
$p \longleftarrow \{\ \}$
\item[] 
$\tilde{q} \longleftarrow \{\ \}$
\item[] 
$E \longleftarrow \left[ \begin {array}{GGG|GGG|GGG|GGG|GGG}
 1&1&1&&&&&&&&&&&&
\\ 0&0&0&1&1&1&&&&&&&&&
\\ 1&0&0&0&0&0&1&1&1&&&&&&
\\ 0&1&0&1&0&0&0&0&0&1&1&1&&&
\\ 1&1&1&0&1&0&1&0&0&0&0&0&1&1&1
\\ &&&1&1&1&0&1&0&1&0&0&0&0&0
\\ &&&&&&1&1&1&0&1&0&1&0&0
\\ &&&&&&&&&1&1&1&0&1&0
\\ &&&&&&&&&&&&1&1&1\end {array}
\right]$
\item[] 
$O \longleftarrow[\ ]$
\item
$j \longleftarrow1$
\item[] Carry out the row operations in $O$ on column $1$. (Nothing to do.)
\item[]
$E \longleftarrow \left[ \begin {array}{BGG|GGG|GGG|GGG|GGG} 
   1&1&1&&&&&&&&&&&&
\\ 0&0&0&1&1&1&&&&&&&&&
\\ 1&0&0&0&0&0&1&1&1&&&&&&
\\ 0&1&0&1&0&0&0&0&0&1&1&1&&&
\\ 1&1&1&0&1&0&1&0&0&0&0&0&1&1&1
\\ &&&1&1&1&0&1&0&1&0&0&0&0&0
\\ &&&&&&1&1&1&0&1&0&1&0&0
\\ &&&&&&&&&1&1&1&0&1&0
\\ &&&&&&&&&&&&1&1&1\end {array}
\right]$
\item[] Identify column $1$ as a pivot.
\item[] 
$p \longleftarrow\{1\}$
\item[] 
$\tilde{q}\longleftarrow\{\ \}$
\item[] Carry out the row operations $(3,-1,1),(5,-1,1)$ on column $1$.
\item[] 
$E\longleftarrow\left[ \begin {array}{BGG|GGG|GGG|GGG|GGG} 
   1&1&1&&&&&&&&&&&&
\\ &0&0&1&1&1&&&&&&&&&
\\ &0&0&0&0&0&1&1&1&&&&&&
\\ &1&0&1&0&0&0&0&0&1&1&1&&&
\\ &1&1&0&1&0&1&0&0&0&0&0&1&1&1
\\ &&&1&1&1&0&1&0&1&0&0&0&0&0
\\ &&&&&&1&1&1&0&1&0&1&0&0
\\ &&&&&&&&&1&1&1&0&1&0
\\ &&&&&&&&&&&&1&1&1\end {array}
\right]$\\
\item[] Append $(3,-1,1),(5,-1,1)$ to $O$.
\item[]
$O\longleftarrow [(3,-1,1),(5,-1,1)]$
\item 
$j\longleftarrow2$
\item[] Carry out the  row operations in $O$ on column $2$.
\item[]
$E \longleftarrow \left[ \begin {array}{BBG|GGG|GGG|GGG|GGG} 
   1&1&1&&&&&&&&&&&&
\\ &0&0&1&1&1&&&&&&&&&
\\ &-1&0&0&0&0&1&1&1&&&&&&
\\ &1&0&1&0&0&0&0&0&1&1&1&&&
\\ &0&1&0&1&0&1&0&0&0&0&0&1&1&1
\\ &&&1&1&1&0&1&0&1&0&0&0&0&0
\\ &&&&&&1&1&1&0&1&0&1&0&0
\\ &&&&&&&&&1&1&1&0&1&0
\\ &&&&&&&&&&&&1&1&1\end {array}
\right]$
\item[] Identify column $2$ as a pivot.
\item[] 
$p\longleftarrow \{1,2\}$
\item[] 
$\tilde{q}\longleftarrow\{\}$
\item[] Carry out the row operations $(3,2),(4,1,2)$ on column $2$.
\item[] 
$E\longleftarrow\left[ \begin {array}{BBG|GGG|GGG|GGG|GGG} 
   1&1&1&&&&&&&&&&&&
\\ &-1&0&1&1&1&&&&&&&&&
\\ &&0&0&0&0&1&1&1&&&&&&
\\ &&0&1&0&0&0&0&0&1&1&1&&&
\\ &&1&0&1&0&1&0&0&0&0&0&1&1&1
\\ &&&1&1&1&0&1&0&1&0&0&0&0&0
\\ &&&&&&1&1&1&0&1&0&1&0&0
\\ &&&&&&&&&1&1&1&0&1&0
\\ &&&&&&&&&&&&1&1&1\end {array}
\right]$
\item[] Append $(3,2),(4,1,2)$ to $O$.
\item[]
$O\longleftarrow[(3,-1,1),(5,-1,1),(3,2),(4,1,2)]$
\item
$j\longleftarrow3$
\item[] Carry out the  row operations in $O$ on column $3$.
\item[]
$E \longleftarrow \left[ \begin {array}{BBB|GGG|GGG|GGG|GGG} 
   1&1&1&&&&&&&&&&&&
\\ &-1&-1&1&1&1&&&&&&&&&
\\ &&0&0&0&0&1&1&1&&&&&&
\\ &&-1&1&0&0&0&0&0&1&1&1&&&
\\ &&0&0&1&0&1&0&0&0&0&0&1&1&1
\\ &&&1&1&1&0&1&0&1&0&0&0&0&0
\\ &&&&&&1&1&1&0&1&0&1&0&0
\\ &&&&&&&&&1&1&1&0&1&0
\\ &&&&&&&&&&&&1&1&1\end {array}
\right]$
\item[] Identify column $3$ as a pivot.
\item[]
$p\longleftarrow\{1,2,3\}$
\item[]
$\tilde{q}\longleftarrow\{\ \}$
\item[] Carry out the row operation $(4,3)$ on column $3$.
\item[]
$E\longleftarrow\left[ \begin {array}{BBB|GGG|GGG|GGG|GGG} 
   1&1&1&&&&&&&&&&&&
\\ &-1&-1&1&1&1&&&&&&&&&
\\ &&-1&0&0&0&1&1&1&&&&&&
\\ &&&1&0&0&0&0&0&1&1&1&&&
\\ &&&0&1&0&1&0&0&0&0&0&1&1&1
\\ &&&1&1&1&0&1&0&1&0&0&0&0&0
\\ &&&&&&1&1&1&0&1&0&1&0&0
\\ &&&&&&&&&1&1&1&0&1&0
\\ &&&&&&&&&&&&1&1&1\end {array}
\right]$
\item[] Append $(4,3)$ to $O$.
\item[]
$O\longleftarrow[(3,-1,1),(5,-1,1),(3,2),(4,1,2),(4,3)]$
\item
$j\longleftarrow4$
\item[] Carry out the  row operations in $O$ on column $4$.
\item[]
$E \longleftarrow \left[ \begin {array}{BBB|BGG|GGG|GGG|GGG} 
   1&1&1&&&&&&&&&&&&
\\ &-1&-1&0&1&1&&&&&&&&&
\\ &&-1&1&0&0&1&1&1&&&&&&
\\ &&&1&0&0&0&0&0&1&1&1&&&
\\ &&&0&1&0&1&0&0&0&0&0&1&1&1
\\ &&&1&1&1&0&1&0&1&0&0&0&0&0
\\ &&&&&&1&1&1&0&1&0&1&0&0
\\ &&&&&&&&&1&1&1&0&1&0
\\ &&&&&&&&&&&&1&1&1\end {array}
\right]$
\item[]  Identify column $4$ as a pivot.
\item[]
$p\longleftarrow\{1,2,3,4\}$
\item[]
$\tilde{q}\longleftarrow\{\ \}$
\item[]  Carry out the row operation $(6,-1,4)$ on column $4$.
\item[]
$E\longleftarrow\left[ \begin {array}{BBB|BGG|GGG|GGG|GGG} 
   1&1&1&&&&&&&&&&&&
\\ &-1&-1&0&1&1&&&&&&&&&
\\ &&-1&1&0&0&1&1&1&&&&&&
\\ &&&1&0&0&0&0&0&1&1&1&&&
\\ &&&&1&0&1&0&0&0&0&0&1&1&1
\\ &&&&1&1&0&1&0&1&0&0&0&0&0
\\ &&&&&&1&1&1&0&1&0&1&0&0
\\ &&&&&&&&&1&1&1&0&1&0
\\ &&&&&&&&&&&&1&1&1\end {array}
\right]$
\item[]  Append $(6,-1,4)$ to $O$.
\item[]
$O\longleftarrow[(3,-1,1),(5,-1,1),(3,2),(4,1,2),(4,3),(6,-1,4)]$
\item
$j\longleftarrow5$
\item[]  Carry out the  row operations in $O$ on column $5$.
\item[]
$E \longleftarrow \left[ \begin {array}{BBB|BBG|GGG|GGG|GGG} 
   1&1&1&&&&&&&&&&&&
\\ &-1&-1&0&0&1&&&&&&&&&
\\ &&-1&1&0&0&1&1&1&&&&&&
\\ &&&1&1&0&0&0&0&1&1&1&&&
\\ &&&&1&0&1&0&0&0&0&0&1&1&1
\\ &&&&0&1&0&1&0&1&0&0&0&0&0
\\ &&&&&&1&1&1&0&1&0&1&0&0
\\ &&&&&&&&&1&1&1&0&1&0
\\ &&&&&&&&&&&&1&1&1\end {array}
\right]$
\item[]  Identify column $5$ as a pivot.
\item[]
$p\longleftarrow\{1,2,3,4,5\}$
\item[]
$\tilde{q}\longleftarrow\{\ \}$
\item[]  No row operations needed on column $5$.
\item[]
$E\longleftarrow\left[ \begin {array}{BBB|BBG|GGG|GGG|GGG} 
   1&1&1&&&&&&&&&&&&
\\ &-1&-1&0&0&1&&&&&&&&&
\\ &&-1&1&0&0&1&1&1&&&&&&
\\ &&&1&1&0&0&0&0&1&1&1&&&
\\ &&&&1&0&1&0&0&0&0&0&1&1&1
\\ &&&&&1&0&1&0&1&0&0&0&0&0
\\ &&&&&&1&1&1&0&1&0&1&0&0
\\ &&&&&&&&&1&1&1&0&1&0
\\ &&&&&&&&&&&&1&1&1\end {array}
\right]$ 
\item[] Nothing to append to $O$.
\item[]
$O\longleftarrow[(3,-1,1),(5,-1,1),(3,2),(4,1,2),(4,3),(6,-1,4)]$
\item
$j\longleftarrow6$
\item[]  Carry out the  row operations in $O$ on column $6$.
\item[]
$E \longleftarrow \left[ \begin {array}{BBB|BBR|GGG|GGG|GGG} 
   1&1&1&&&&&&&&&&&&
\\ &-1&-1&0&0&0&&&&&&&&&
\\ &&-1&1&0&0&1&1&1&&&&&&
\\ &&&1&1&1&0&0&0&1&1&1&&&
\\ &&&&1&0&1&0&0&0&0&0&1&1&1
\\ &&&&&0&0&1&0&1&0&0&0&0&0
\\ &&&&&&1&1&1&0&1&0&1&0&0
\\ &&&&&&&&&1&1&1&0&1&0
\\ &&&&&&&&&&&&1&1&1\end {array}
\right]$
\item[]  Identify column $6$ as a basic non-pivot: column 6 is non-pivotal  because  it does not have non-zero entries below the 5-th row and therefore it is a linear combination of the five previous pivotal columns:   $E_{*6}=-E_{*2}+E_{*3}+E_{*4}$. Column 6 is basic because its index is minimal in  its equivalence class $q/(3)$.
\item[]
$p\longleftarrow\{1,2,3,4,5\}$
\item[]
$\tilde{q}\longleftarrow\{6\}$
\item[] No row operations needed on column $6$.
\item[]
$E\longleftarrow\left[ \begin {array}{BBB|BBR|GGG|GGG|GGG} 
   1&1&1&&&&&&&&&&&&
\\ &-1&-1&0&0&0&&&&&&&&&
\\ &&-1&1&0&0&1&1&1&&&&&&
\\ &&&1&1&1&0&0&0&1&1&1&&&
\\ &&&&1&0&1&0&0&0&0&0&1&1&1
\\ &&&&&&0&1&0&1&0&0&0&0&0
\\ &&&&&&1&1&1&0&1&0&1&0&0
\\ &&&&&&&&&1&1&1&0&1&0
\\ &&&&&&&&&&&&1&1&1\end {array}
\right]$ 
\item[] Nothing to append to $O$.
\item[]
$O\longleftarrow[(3,-1,1),(5,-1,1),(3,2),(4,1,2),(4,3),(6,-1,4)]$

\item
$j\longleftarrow7$
\item[]  Carry out the  row operations in $O$ on column $7$.
\item[]
$E \longleftarrow \left[ \begin {array}{BBB|BBR|BGG|GGG|GGG} 
   1&1&1&&&&&&&&&&&&
\\ &-1&-1&0&0&0&1&&&&&&&&
\\ &&-1&1&0&0&1&1&1&&&&&&
\\ &&&1&1&1&0&0&0&1&1&1&0&0&0
\\ &&&&1&0&1&0&0&0&0&0&1&1&1
\\ &&&&&&0&1&0&1&0&0&0&0&0
\\ &&&&&&1&1&1&0&1&0&1&0&0
\\ &&&&&&&&&1&1&1&0&1&0
\\ &&&&&&&&&&&&1&1&1\end {array}
\right]$
\item[]  Identify column $7$ as a pivot.
\item[]
$p\longleftarrow\{1,2,3,4,5,7\}$
\item[]
$\tilde{q}\longleftarrow\{6\}$
\item[]  Carry out the row operations $(7,6)$ on column $7$.
\item[]
$E\longleftarrow\left[ \begin {array}{BBB|BBR|BGG|GGG|GGG} 
   1&1&1&&&&&&&&&&&&
\\ &-1&-1&0&0&0&1&&&&&&&&
\\ &&-1&1&0&0&1&1&1&&&&&&
\\ &&&1&1&1&0&0&0&1&1&1&&&
\\ &&&&1&0&1&0&0&0&0&0&1&1&1
\\ &&&&&&1&1&0&1&0&0&0&0&0
\\ &&&&&&&1&1&0&1&0&1&0&0
\\ &&&&&&&&&1&1&1&0&1&0
\\ &&&&&&&&&&&&1&1&1\end {array}
\right]$ 
\item[]  Append $(7,6)$ to $O$.
\item[]
$O\longleftarrow[(3,-1,1),(5,-1,1),(3,2),(4,1,2),(4,3),(6,-1,4),(7,6)]$
\item
$j\longleftarrow8$
\item[]  Carry out the  row operations in $O$ on column $8$.
\item[]
$E \longleftarrow \left[ \begin {array}{BBB|BBR|BBG|GGG|GGG} 
   1&1&1&&&&&&&&&&&&
\\ &-1&-1&0&0&0&1&1&&&&&&&
\\ &&-1&1&0&0&1&1&1&&&&&&
\\ &&&1&1&1&0&0&0&1&1&1&&&
\\ &&&&1&0&1&0&0&0&0&0&1&1&1
\\ &&&&&&1&1&0&1&0&0&0&0&0
\\ &&&&&&&1&1&0&1&0&1&0&0
\\ &&&&&&&&&1&1&1&0&1&0
\\ &&&&&&&&&&&&1&1&1\end {array}
\right]$
\item[]  Identify column $8$ as a pivot.
\item[]
$p\longleftarrow\{1,2,3,4,5,7,8\}$
\item[]
$\tilde{q}\longleftarrow\{6\}$
\item[] No row operations needed on column $8$.
\item[]
$E\longleftarrow\left[ \begin {array}{BBB|BBR|BBG|GGG|GGG} 
   1&1&1&&&&&&&&&&&&
\\ &-1&-1&0&0&0&1&1&&&&&&&
\\ &&-1&1&0&0&1&1&1&&&&&&
\\ &&&1&1&1&0&0&0&1&1&1&&&
\\ &&&&1&0&1&0&0&0&0&0&1&1&1
\\ &&&&&&1&1&0&1&0&0&0&0&0
\\ &&&&&&&1&1&0&1&0&1&0&0
\\ &&&&&&&&&1&1&1&0&1&0
\\ &&&&&&&&&&&&1&1&1\end {array}
\right]$
\item[]  Nothing to append to $O$.
\item[]
$O\longleftarrow[(3,-1,1),(5,-1,1),(3,2),(4,1,2),(4,3),(6,-1,4),(7,6)]$
\item
$j\longleftarrow9$
\item[]  Identify column $9$ as periodic non-pivot.
\item[]
$E\longleftarrow\left[ \begin {array}{BBB|BBR|BBT|GGG|GGG} 
   1&1&1&&&&&&&&&&&&
\\ &-1&-1&0&0&0&1&1&&&&&&&
\\ &&-1&1&0&0&1&1&1&&&&&&
\\ &&&1&1&1&0&0&0&1&1&1&&&
\\ &&&&1&0&1&0&0&0&0&0&1&1&1
\\ &&&&&&1&1&0&1&0&0&0&0&0
\\ &&&&&&&1&1&0&1&0&1&0&0
\\ &&&&&&&&&1&1&1&0&1&0
\\ &&&&&&&&&&&&1&1&1\end {array}
\right]$ 
\item
$j\longleftarrow10$
\item[]  Carry out the  row operations in $O$ on column $10$.
\item[]
$E \longleftarrow \left[ \begin {array}{BBB|BBR|BBT|BGG|GGG} 
   1&1&1&&&&&&&&&&&&
\\ &-1&-1&0&0&0&1&1&&0&&&&&
\\ &&-1&1&0&0&1&1&1&1&&&&&
\\ &&&1&1&1&0&0&0&0&1&1&&&
\\ &&&&1&0&1&0&0&0&0&0&1&1&1
\\ &&&&&&1&1&0&0&0&0&0&0&0
\\ &&&&&&&1&1&1&1&0&1&0&0
\\ &&&&&&&&&1&1&1&0&1&0
\\ &&&&&&&&&&&&1&1&1\end {array}
\right]$
\item[]  Identify column $10$ as a pivot.
\item[]
$p\longleftarrow\{1,2,3,4,5,7,8,10\}$
\item[]
$\tilde{q}\longleftarrow\{6\}$
\item[] No row operations needed on column $10$.
\item[]
$E\longleftarrow\left[ \begin {array}{BBB|BBR|BBT|BGG|GGG} 
   1&1&1&&&&&&&&&&&&
\\ &-1&-1&0&0&0&1&1&&0&&&&&
\\ &&-1&1&0&0&1&1&1&1&&&&&
\\ &&&1&1&1&0&0&0&0&1&1&&&
\\ &&&&1&0&1&0&0&0&0&0&1&1&1
\\ &&&&&&1&1&0&0&0&0&0&0&0
\\ &&&&&&&1&1&1&1&0&1&0&0
\\ &&&&&&&&&1&1&1&0&1&0
\\ &&&&&&&&&&&&1&1&1\end {array}
\right]$
\item[]  Nothing to append to $O$.
\item[]
$O\longleftarrow[(3,-1,1),(5,-1,1),(3,2),(4,1,2),(4,3),(6,-1,4),(7,6)]$
\item
$j\longleftarrow11$
\item[]  Carry out the  row operations in $O$ on column $11.$
\item[]
$E \longleftarrow \left[ \begin {array}{BBB|BBR|BBT|BRG|GGG} 
   1&1&1&&&&&&&&&&&&
\\ &-1&-1&0&0&0&1&1&&0&0&&&&
\\ &&-1&1&0&0&1&1&1&1&1&&&&
\\ &&&1&1&1&0&0&0&0&0&1&&&
\\ &&&&1&0&1&0&0&0&0&0&1&1&1
\\ &&&&&&1&1&0&0&1&0&0&0&0
\\ &&&&&&&1&1&1&0&0&1&0&0
\\ &&&&&&&&&1&1&1&0&1&0
\\ &&&&&&&&&&&&1&1&1\end {array}
\right]$
\item[] Identify column $11$ as a basic non-pivot.
\item[]
$p\longleftarrow\{1,2,3,4,5,7,8,10\}$
\item[]
$\tilde{q}\longleftarrow\{6,11\}$
\item[] No row operations needed on column $11$.
\item[]
$E\longleftarrow\left[ \begin {array}{BBB|BBR|BBT|BRG|GGG} 
   1&1&1&&&&&&&&&&&&
\\ &-1&-1&0&0&0&1&1&&0&0&&&&
\\ &&-1&1&0&0&1&1&1&1&1&&&&
\\ &&&1&1&1&0&0&0&0&0&1&&&
\\ &&&&1&0&1&0&0&0&0&0&1&1&1
\\ &&&&&&1&1&0&0&1&0&0&0&0
\\ &&&&&&&1&1&1&0&0&1&0&0
\\ &&&&&&&&&1&1&1&0&1&0
\\ &&&&&&&&&&&&1&1&1\end {array}
\right]$
\item[]  Nothing to append to $O$.
\item[]
$O\longleftarrow[(3,-1,1),(5,-1,1),(3,2),(4,1,2),(4,3),(6,-1,4),(7,6)]$
\item[]
We have identified $n-1$ basic non-pivot columns, so we abort forward elimination.
\item Perform backward elimination on the pivot columns and basic non-pivot columns.
\item[]
$E \longleftarrow \left[ \begin {array}{BBB|BBR|BBT|BRG|GGG} 
   1&&&&&0&&&&&-1&&&&
\\ &-1&&&&1&&&&&2&&&&
\\ &&-1&&&-1&&&1&&-3&&&&
\\ &&&1&&1&&&0&&2&1&&&
\\ &&&&1&0&&&0&&-2&0&1&1&1
\\ &&&&&&1&&0&&2&0&0&0&0
\\ &&&&&&&1&1&&-1&0&1&0&0
\\ &&&&&&&&&1&1&1&0&1&0
\\ &&&&&&&&&&&&1&1&1\end {array}
\right]$ 
\item Perform normalization on the pivot columns and basic non-pivot columns.
\item[]
$E\longleftarrow\left[ \begin {array}{BBB|BBR|BBT|BRG|GGG} 
   1&&&&&0&&&&&-1&&&&
\\ &1&&&&-1&&&&&-2&&&&
\\ &&1&&&1&&&1&&3&&&&
\\ &&&1&&1&&&0&&2&1&&&
\\ &&&&1&0&&&0&&-2&0&1&1&1
\\ &&&&&&1&&0&&2&0&0&0&0
\\ &&&&&&&1&1&&-1&0&1&0&0
\\ &&&&&&&&&1&1&1&0&1&0
\\ &&&&&&&&&&&&1&1&1\end {array}
\right]$ 
\end{enumerate}

\item \ \emph {Construct a matrix $M\in\k[s]^{n\times\left(  n-1\right)  } $ whose columns form a $\mu$-basis of $a$}:
\begin{enumerate}
\item $M \longleftarrow \left[
\begin{array}[c]{cc}
0 & 0\\
0 & 0\\
0 & 0
\end{array}
\right]$

\item $M \longleftarrow \left[
\begin{array}[c]{cc}
0 & 0\\
0 & s^3\\
s & 0
\end{array}
\right]$
\item $M\longleftarrow \left[
\begin{array}
[c]{cc}%
-s & 1-2s-2s^{2}-s^{3}\\
1 & 2+2s+s^{2}+s^{3}\\
-1+s & -3
\end{array}
\right]
$
\end{enumerate}
\end{enumerate}
\end{example}

\section{Theoretical Complexity Analysis}
\label{sect-compl}

In this subsection, we analyze the theoretical (asymptotic worst case) complexity of the $\mu$-basis algorithm given in the previous section. We will do so under the assumption that the time for any arithmetic operation is constant. 

\begin{theorem}\label{thm-complexity}
The complexity of the $\mu$-basis algorithm given in the previous section is
$$
O(d^2n + d^3 +  n^2).
$$
\end{theorem}

\begin{proof}
We will trace the theoretical complexity for each step of the algorithm.

\begin{enumerate}
\item
\begin{enumerate}
\item To determine $d$, we scan through each of the $n$ polynomials in $a$ to
identify the highest degree term, which is always $\leq d$. Thus, the complexity
for this step is $O(dn)$.

\item We identify $n(d+1)$ values to make up $c_{0},\ldots,c_{d}$. Thus, the
complexity for this step is $O(dn)$.

\item We construct a matrix with $(2d+1)n(d+1)$ entries. Thus, the complexity for
this step is $O(d^{2}n)$.
\end{enumerate}

\item With the partial Gauss-Jordan elimination, we  perform row operations only on the (at most) $2d+1$ pivot columns of $A$ and the $n-1$ basic non-pivot columns of $A$. Thus, we perform Gauss-Jordan elimination on a $(2d+1)\times (2d+n)$ matrix.
In general, for a $k\times l$ matrix, Gauss-Jordan elimination has complexity $O(k^{2}%
l)$. Thus, the complexity for this step is $O(d^{2}(d+n))$. 

\item 
\begin{enumerate}
\item We  fill 0 into the  entries of an $n\times (n-1)$ matrix $M$. 
Thus, the complexity of this step is $O(n^2)$.
\item We update   entries of the matrix $n-1$ times. Thus, the complexity of this step is $O(n)$.
\item We update  entries of the matrix $|p| \times (n-1)$ times. Note that $|p| = \text{rank}(A) \le 2d+1$.
Thus
 the complexity of this step is $O(dn)$.
\end{enumerate}

\end{enumerate}

\noindent By summing up, we have%
\[
O\left(  dn+dn+d^{2}n+d^{2}(d+n)+n^2+n+dn\right)  
 =O\left(  d^2n+d^{3}+n^2 \right)
\]
\end{proof}

\begin{remark}\label{rem-sparse}\rm
Note that the $n^2$  term in the above complexity is solely due to step 3(a), where the matrix $M$ is initialized with zeros. If one uses a {\em sparse\/} representation of the matrix (storing only non-zero elements),
then one can skip the initialization of the matrix $M$. As a result, the complexity can be improved to  
$O\left(d^2n+d^3\right)$.
\end{remark}

%We wish to compare the $\mu$-Basis Algorithm in this paper to the algorithm presented in \cite{song-goldman-2009}, which we reprint here for the readers' convenience.

\begin{remark}\label{rem-sg-complexity}\rm\rm[Comparison with Song-Goldman Algorithm]
As far as we are aware, the theoretical complexity of the algorithm by   Song and Goldman \cite{song-goldman-2009} has  not yet been published. 
Here we roughly estimate the complexity of this algorithm to be  $O(dn^5+d^2n^4)$.  It will require a more  rigorous analysis to prove/refute this apparent complexity, which is beyond the scope of this paper.   
For the readers' convenience, we reproduce the  algorithm  published   in  \cite{song-goldman-2009} on pp. 220 -- 221
in our notation, before analyzing its complexity.

\rm $\text{   }$ \\
Input: $a \in \mathbb{K}[s]^n$ with $\gcd(a)=1$\\
Output: A $\mu$-basis of $a$
\begin{enumerate}
\item Create the $r = C^n_2$ ``obvious" syzygies as described in Lemma \ref{lem-d} and label them $u_1,\ldots,u_r$.
\item Set $m_i = LV(u_i)$ and $d_i = \deg(u_i)$ for $i=1,\ldots,r$.
\item Sort $d_i$ so that $d_1 \ge d_2 \ge \ldots \ge d_r$ and re-index $u_i$, $m_i$.
\item Find real numbers $\alpha_1,\alpha_2,\ldots,\alpha_r$ such that $\alpha_1m_1+\alpha_2m_2+\cdots+\alpha_rm_r = 0$.
\item Choose the lowest integer $j$ such that $\alpha_j \not= 0$, and update $u_j$ by setting $$u_j = \alpha_ju_j + \alpha_{j+1}s^{d_j-d_{j+1}}u_{j+1} + \cdots + \alpha_rs^{d_j-d_r}u_r.$$ If $u_j \equiv 0$, discard $u_j$ and set $r=r-1$; otherwise set $m_j = LV(u_j)$ and $d_j = \deg(u_j)$.
\item If $r = n-1$, then output the $n-1$ non-zero vector polynomials $u_1,\ldots,u_{n-1}$ and stop; otherwise, go to Step 3.
\end{enumerate}

\noindent Finding a null vector in step 4  by partial Gauss-Jordan elimination requires performing row operations on (at most) $n+1$ columns. Since each column contains $n$ entries, we conclude that this step has complexity $O(n^3)$. Performing the ``update" operation in step 5 of the algorithm has complexity $O(dn^2)$. Step 6 implies that, in the worst case, the algorithm repeats steps 4 and 5 at most  $\left(d\,\frac  {n(n-1)}2-d\right)$ times. The reason is as follows. Since the algorithm starts with the $C^n_2 = \frac{n(n-1)}{2}$ obvious syzygies  and each has   degree $\leq d$, the (worst case) total degree of the syzygies at the beginning of the algorithm is $d\, \frac{n(n-1)}{2}$. The algorithm ends only when the total degree is $d$. If each repetition of steps 4 and 5 reduces the total degree by 1, then the steps are repeated $\left(d\,\frac  {n(n-1)}2-d\right)$ times. Thus, the total computational complexity appears  to be $O(dn^5+d^2n^4)$.\end{remark}

%%%%%%
\section{Implementation, experiments, comparison}
\label{sect-experiment}
We implemented  the $\mu$-basis algorithm presented in this paper and the one described in Song-Goldman \cite{song-goldman-2009}.  For the sake of simplicity, from now on, we will
call these two algorithms HHK and SG. In Section~\ref{sect-impl}, we discuss our implementation. In Section~\ref{sect-time}, we describe the  experimental performance of both algorithms. An experimental timing corresponds to a point $(d,n,t)$, where $d$ is the degree,  $n$  is the length of the input polynomial vector,  and  $t$ is the time in seconds it took for our codes to produce the output.  For each algorithm,  we fit a surface through the experimental data points. Our fitting models are based on the theoretical complexities obtained in  Section~\ref{sect-compl}. In   Section~\ref{sect-comp},  we compare the performance of the two algorithms.
%%%
\subsection{Implementation}\label{sect-impl}
We implemented both algorithms (HHK and SG) in the computer algebra system Maple~\cite{maple}. The codes and examples are available on the web:
\begin{center}
\url{http://www.math.ncsu.edu/~zchough/mubasis.html} 
\end{center}
We post two versions of the code:
\begin{itemize}
\item[] \texttt{program\_rf} : over rational number field $\mathbb{Q}$.
\item[] \texttt{program\_ff} : over finite field $\mathbb{F}_p$ where $p$ is an arbitrary prime number.
\end{itemize}
Now we explain how the two algorithms (HHK\ and SG) have been implemented.
\begin{itemize}
\item  Although both algorithms could be written in a non-interpreted language such as the   C-language, making the running time significantly shorter,
we  implemented both algorithms in Maple~\cite{maple} for the following reasons.
\begin{enumerate}
\item Maple allows fast prototyping  of the algorithms, making it easier to write and  read the programs written in Maple.

\item It is expected that potential applications of $\mu$-bases will often be written in computer algebra systems such as Maple. 
\end{enumerate}

 \item Both algorithms contain a step in which null vectors are computed (step 2 of HHK and step 4 of SG). 
Although Maple has  a built-in routine for  computing  a basis of the null space for the input matrix, we did not use this built-in routine because we do not need the entire null basis, but only a certain subset of basis vectors with desired properties, consisting of  $n-1$ vectors for HHK   and a single vector for SG.
For this reason, we implemented partial Gauss-Jordan elimination.

\item For the rational field implementation of the SG algorithm, we produced the null vector in step 4 with integer entries in order to avoid rational number arithmetic (which is expensive due to integer gcd computations) in the subsequent steps of the algorithm.

\item  Dense representations of matrices were used for both algorithms.  As shown in  Remark~\ref{rem-sparse}, it is easy to exploit sparse representations for HHK, but it was not clear how one could exploit   sparse representations for SG.  Thus, in order to ensure fair comparison, we used dense representations for both algorithms.

\end{itemize}

%%%%%
%\subsection*{Setup}

\subsection{Timing and fitting}\label{sect-time}
%In Figure~\ref{fig-timings}, we report the timings.
We explain the setup for our experiments
so that the timings reported here can be reproduced independently.
\begin{itemize}
\item The programs were executed using Maple 2015 version 
running on Apple iMac  (Intel i 7-2600, 3.4 GHz, 16GB).
\item  The inputs were randomly-generated: for various values of $d$ and $n$, the coefficients were taken randomly  from  $\mathbb{F}_5$, with a uniform distribution.
\item In order to get reliable timings, especially when the computing time is small relative to the clock resolution, we ran each program several times on the same input  and computed the average of the computing times.   
\item The execution of a program on an input was cut off if its  computing time  exceeded  120 seconds.
\end{itemize}

\begin{figure}[h!] \centering
\begin{minipage}[b]{0.45\linewidth} \centering
\includegraphics[scale=0.45]{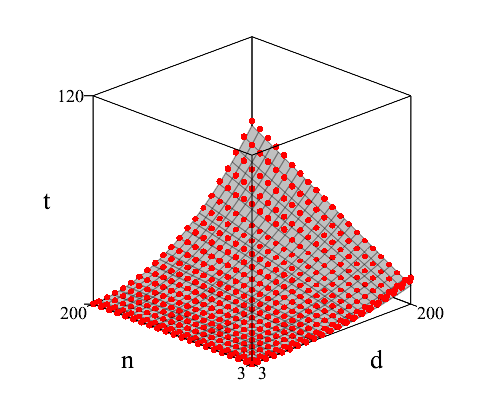} \caption{HHK algorithm timing}
\label{fig-HHK}
\end{minipage} \begin{minipage}[b]{0.45\linewidth}
\centering \includegraphics[scale=0.45]{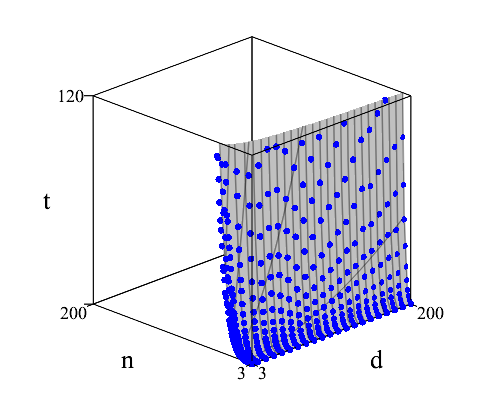}
\caption{SG algorithm timing}
\label{fig-SG}
\end{minipage}
\end{figure}

\noindent Figure~\ref{fig-HHK} shows the experimental timing for the HHK algorithm, while Figure~\ref{fig-SG} shows the experimental timing for the SG algorithm. The algorithms were run on randomly-generated examples with specified $d$ and $n$, and they ran in time $t$.
For each figure, the axes represent the range of values $d = 3,\ldots,200$, \hskip2mm $n = 3,\ldots,200$, and $t = 0,\ldots,120$, where $t$ is the timing  in seconds.
Each dot  $(d,n,t)$ represents an  experimental timing.

%%%%%%%

The background gray surfaces are fitted  to the experimental data. The  fitting  model is based on  the theoretical complexities from Section~\ref{sect-compl}.
 The fitting was computed using least squares. For HHK,  based on Theorem~\ref{thm-complexity}, we chose a model  for the timing, $t= \alpha_1d^2n+\alpha_2d^3+\alpha_3n^2$,  
where $\alpha$'s are unknown constants to be determined.
After substituting  the experimental  values $(d,n,t)$, we obtain an over-determined  system of linear equations in the $\alpha$'s. We find $\alpha$'s that minimize the sum of squares of errors.
For SG, we used the same procedure with  the timing model    $t= \beta_1dn^5+\beta_2d^2n^4$ based on  Remark~\ref{rem-sg-complexity}.

We generated the following functions:
 \begin{eqnarray}
\label{eq-tHHK} t_{HHK} &\approx& 10^{-6}\cdot(7.4\ {d}^{2}n+ 1.2 \ {d}^{3}+ 1.2\ {n}^{2}) 
\\ 
\label{eq-tSG} t_{SG}  &\approx& 10^{-7}\cdot(2.6\ d{n}^{5}+ 0.6\ {d}^{2}{n}^{4}) 
\end{eqnarray}
For our experimental data, the residual standard deviation for  the HHK-timing model \eqref{eq-tHHK}  is 0.686 seconds, while   the residual standard deviation for  the SG-timing model   \eqref{eq-tSG}  is 11.886  seconds.

We observe from Figures~\ref{fig-HHK} and \ref{fig-SG} that for a fixed $d$, the HHK algorithm's dependence on $n$ is almost linear, while the SG algorithm's dependence on $n$ is highly nonlinear. %(for $n \le 200$).
 In fact, for the latter, the dependence is so steep that the algorithm was unable to terminate in under 120 seconds for most values of $n$, thus explaining the large amount of blank space in  Figure \ref{fig-SG}. For a fixed $n$, the HHK algorithm's dependence on $d$ is nonlinear, while the SG algorithm's dependence on $d$ is almost linear. 
 
%The  experimental observations are supported by the fitting formulae for $t_{HHK}$ and $t_{SG}$. For the HHK timing formula, the largest coefficient appears on the $d^2\,n$ term, while for the SG timing formula, the largest coefficient appears on the $d\,n^5$ term. 
 
%%%%%
\subsection{Comparison}\label{sect-comp}
Two pictures below represent performance comparisons.
\begin{figure}[h!] \centering
\begin{minipage}[b]{0.45\linewidth} \centering
\includegraphics[scale=0.45]{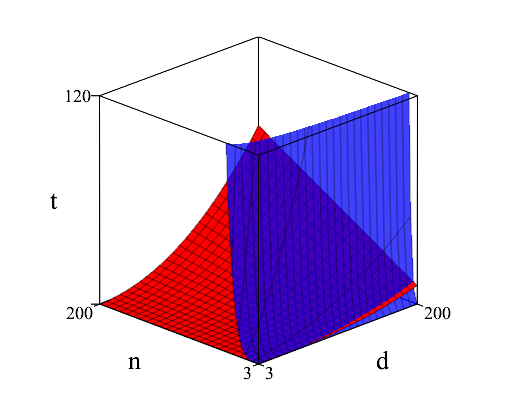} \caption{HHK (red) and SG (blue).}
\label{fig-3d-comp}
\end{minipage} \begin{minipage}[b]{0.45\linewidth}
\centering \includegraphics[scale=0.45]{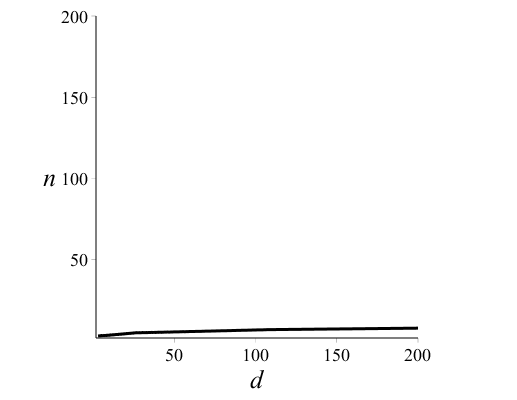}
\caption{Tradeoff graph}
\label{fig-2d-comp}
\end{minipage}
\end{figure}

%In Figure~\ref{fig-compare}, we compare the timings.
%\begin{figure}[h!]
%\begin{center}
%\begin{tabular}{cc}
%\includegraphics[scale=1.5]{plot_HHK_SG.png} &
%\includegraphics[scale=1.5]{plot_tradeoff.png} 
%\end{tabular}
%\end{center}
%\caption{Comparison: HHK (red) and SG (blue)}
%\label{fig-compare}
%\end{figure}

\begin{itemize}
\item Figure~\ref{fig-3d-comp} shows the fitted surfaces from Figures~\ref{fig-HHK} and~\ref{fig-SG} on the same graph.
%  for the HHK (in red) and SG (in blue) algorithms in the same graph. 
The axes represent the range of values $n = 3,\ldots,200$, $d = 3,\ldots,200$, and $t = 0,\ldots,120$, where $t$ is the timing of the algorithms in seconds.
\item Figure~\ref{fig-2d-comp} shows a tradeoff graph for the two algorithms. The curve in the figure represents values of $d$ and $n$ for which the two algorithms run with the same timing. Below the curve, the SG algorithm runs faster, while above the curve, the HHK algorithm runs faster. The ratio of the dominant terms in the fitted  formulae  is $d : n^4$. This ratio manifests itself  in the shape of the tradeoff curve presented in Figure \ref{fig-2d-comp}.

\end{itemize}

\begin{figure}[h!] \centering
\begin{minipage}[b]{0.45\linewidth} \centering
\includegraphics[scale=0.35]{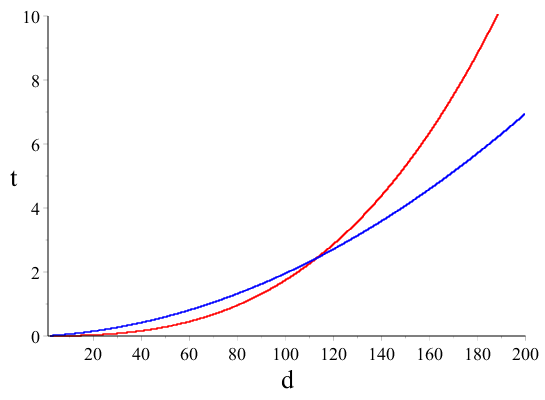} \caption{ $n=7$}
\label{fig-n=7}
\end{minipage} \begin{minipage}[b]{0.45\linewidth}
\centering \includegraphics[scale=0.35]{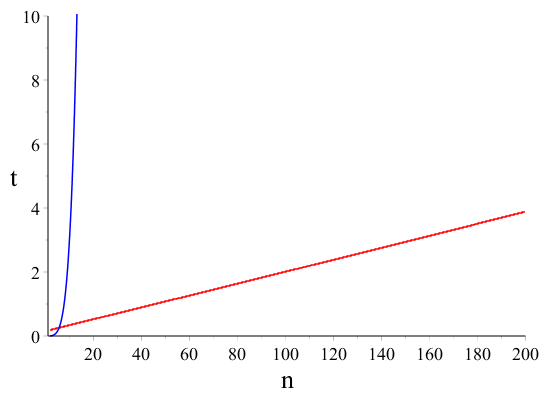}
\caption{{$d=50$}}
\label{fig-d=30}
\end{minipage}
\end{figure}

From Figure \ref{fig-3d-comp}, we observe that for a fixed $d$, as $n$ increases the HHK algorithm vastly outperforms the SG algorithm.
In contrast, for a  fixed value of $n$, as $d$ increases the SG algorithm outperforms the HHK algorithm.
 {The order by which SG runs faster is less than the order by which HHK runs faster for fixed $d$ and increasing $n$}.
 We underscore this observation by displaying  two-dimensional slices of Figure \ref{fig-3d-comp}.  Figure~\ref{fig-n=7} represents the slice in the $d$-direction  with $n=7$, while Figure~\ref{fig-d=30}
  represents the slice in the $n$-direction with $d=50$. As before, HHK is represented by red and SG by blue.
  
 %%%%
 \section{Discussion}\label{sect-discussion}
 %%%
 In this section, we elaborate on some topics that were briefly discussed in the Introduction and Section~\ref{sect-problem}  and discuss  a  natural generalization of the $\mu$-basis  computation problem   --  a problem of computing minimal  bases of the kernels of $m\times n$ polynomial matrices.

%%%

\vskip2mm \noindent\emph{The original definition and proof of existence:} 
The original definition of a $\mu$-basis appeared on p 824 of a paper by Cox, Sederberg, and Chen  \cite{cox-sederberg-chen-1998} and is  based on  the ``sum of the degrees'' property (Statement~\ref{pr-min} of Proposition~\ref{prop-equiv}). The definition also mentions an equivalent  ``reduced  representation'' (Statement~\ref{pr-reduced} of Proposition~\ref{prop-equiv}). The proof of the existence theorem (Theorem 1 on p.~824 of \cite{cox-sederberg-chen-1998}) appeals to the celebrated Hilbert Syzygy Theorem \cite{hilbert-1890}  and utilizes Hilbert polynomials, which  first appeared in the same paper  \cite{hilbert-1890}    under the name of characteristic functions. The definition of $\mu$-basis in terms of the degrees, given in \cite{cox-sederberg-chen-1998}, is compatible with the tools that have been chosen to show its  existence.

%%%%%%%%%%
\vskip2mm \noindent\emph{The homogeneous version of the problem:} 
  It is instructive to compare the inhomogeneous and homogenous versions of the problem. In fact, in order to invoke the Hilbert Syzygy Theorem in the proof of the existence of a $\mu$-basis, Cox, Sederberg, and Chen  restated the problem in the homogeneous setting (see pp.~824-825 of \cite{cox-sederberg-chen-1998}).  

Let $\hat a =[\hat a_1(x,y), \dots,\hat a_n(x,y)]$ be a row vector of $n$ homogeneous polynomials over a field $\k$, each of which has the same degree.  
As before, a syzygy of $\hat a$ is a column vector  $h= [h_1(x,y), \dots,h_n(x,y)]^T$ of polynomials (not necessarily homogeneous), such that $\hat a\,h=0$. The set $\syz(\hat a)$ is a module over $\k[x,y]$,
   and the Hilbert Syzygy Theorem implies that it is a free module of rank $n-1$ possessing a homogeneous basis. Let $n-1$  homogeneous polynomial vectors $\hat u_1(x,y),\dots,\hat u_{n-1}(x,y)$  comprise an \emph{arbitrary}  homogeneous basis of $\syz(\hat a)$. Define dehomogenizations:  $a(s)=[a_1(s),\dots, a_n(s)]$, where $a_i(s)=\hat a_i(s,1)\in\k[s]$, $i=1,\dots, n$ and  $u_j(s)=\hat u_j (s,1)\in\k[s]^n$, $j=1,\dots, n-1$. An argument, involving Hilbert polynomials  on  p. 825 of \cite{cox-sederberg-chen-1998}, shows that $u_1,\dots, u_{n-1}$ is a \emph{$\mu$-basis} of $\syz(a)$.  
   
     Let us now start   with a polynomial vector $a(s)=[a_1(s),\dots, a_n(s)]\in\k[s]^n$ of degree $d$ in the sense of Definition~\ref{def-basic}, and consider its homogenization $\hat a =[\hat a_1(x,y), \dots,\hat a_n(x,y)]$, where $\hat a_i(x,y)=y^d\,a_i\left(\frac x y\right)$,  $i=1,\dots, n$. It is not true that homogenezation of an arbitrary basis  of $\syz(a)$ produces a basis of $\syz(\hat a)$. Indeed, 
     let $u_1$ and $u_2$ be the columns of matrix $M$ in Example~\ref{ex-prob}. Then $u_1+u_2$ and $u_2$ is a basis of $\syz(a)$, with each vector having degree 3.
Their  homogenizations $\widehat{ u_1 +u_2}$ and $\widehat{ u_2}$ are homogeneous polynomial vectors of degree 3, and, therefore, they can not possibly generate a homogeneous  vector 
$\hat u_1(x,y)=y\, u_1\left(\frac xy \right)=[-x,\,y,\,x-y]^T$ of degree 1, which clearly belongs to $\syz(\hat a)$. 
A rather simple argument that utilizes the ``reduced representation'' property  (Statement~\ref{pr-reduced}  of Proposition~\ref{prop-equiv}) can be used to show that for an arbitrary non-zero vector $a\in\k[s]^n$, homogenization of any \emph{$\mu$-basis} of $\syz(a)$ produces a homogeneous basis of $\syz(\hat a )$.

The above discussion can be summarized in the following statement:  the set of homogeneous bases of  $\syz(\hat a)$ is in one-to-one correspondence with the set of $\mu$-bases of $\syz(a)$, where $a\in\k[s]$  is the dehomogenization of  $\hat a\in\k[x,y]$. Therefore, the algorithm developed in this paper can be used to compute homogeneous bases of  $\syz(\hat a)$.

%%%%%
\vskip2mm \noindent\emph{$\mu$-basis algorithms and  $\gcd$ computation:} 
In contrast to the algorithm developed by  Song and Goldman in \cite{song-goldman-2009},  the algorithm presented  in this paper produces  a $\mu$-basis even when the input vector $a$ has a non-trivial     greatest common divisor. Moreover, once a $\mu$-basis is computed, one can \emph{immediately} find $\gcd(a)$ using Statement~\ref{pr-outer} of Proposition~\ref{prop-equiv}. Indeed,  let
$h$ denote the outer product of a $\mu$-basis     $u_1,\dots, u_{n-1}$. If $M$ is the matrix generated by the algorithm, then $h_i= 
(-1)^i \,|M_i|$, where $M_i$ is an $(n-1)\times (n-1)$ submatrix of $M$ constructed by removing the $i$-th row. By Statement~\ref{pr-outer} of Proposition~\ref{prop-equiv},   there exists a non-zero $\alpha\in\k$ such that
$$ a= \alpha \,\gcd(a)\,h.$$
Let $i\in \{1,\dots, n\}$ be such that $a_i$ is a non-zero polynomial. Then $ \gcd(a)$ is computed by  long division of $a_i$ by $ h_i$ and then dividing the quotient  by its leading coefficient to make it monic.
In comparison, the  algorithm developed in \cite{song-goldman-2009} produces a $\mu$-basis of $a$ multiplied by $\gcd(a)$. From the output of this algorithm and  Statement~\ref{pr-outer} of Proposition~\ref{prop-equiv}, one finds  $\gcd(a)^{n-2}$. Song and Goldman discuss how to  recover $\gcd(a)$ itself by \emph{repeatedly} running their algorithm. They also run computational experiments to compare the efficiency of computing $\gcd$ by iterating the SG $\mu$-basis algorithm versus the standard Euclidean algorithm. Investigation of the efficiency of computing   $\gcd$ by using the HHK  $\mu$-basis algorithm and a long division can be a subject of a future  work.

%$$ \syz(\hat a)=\widehat{\syz(a)}:=$$

%%%%%%%%%%
\vskip2mm \noindent\emph{Kernels of $m\times n$ polynomial matrices:} A natural generalization of the $\mu$-basis problem is obtained by  considering  kernels, or nullspaces, of $m\times n$ polynomial matrices of rank $m$.
% It can be proved that, in one variable case, the kernel is a free module of rank $n-m$.
A basis of the nullspace is called minimal if  the ``minimal degree'' Statement~\ref{pr-min} of Proposition~\ref{prop-equiv} is satisfied (with $n-1$ replaced by $n-m$).   One can easily adapt the argument in the proof of Theorem~2 in \cite{song-goldman-2009} to show that, in this more general setting,   Statement~\ref{pr-min} is equivalent to the ``independence of the leading vectors'' Statement~\ref{pr-LV} and  to the ``reduced representation'' Statement~\ref{pr-reduced}  of Proposition~\ref{prop-equiv}. One can also show with an example that the ``sum of the degrees'' Statement~\ref{pr-sum}  (with the degree of a polynomial matrix defined to be the  maximum of the degrees of its entries) is no longer equivalent to Statements~\ref{pr-LV} and~\ref{pr-reduced}. There is a large body of  work on computing minimal  bases (see for instance \cite{beelen1987},  \cite{antoniou2005}, \cite{zhou-2012} and references therein). This research direction  seems to be developing independently of the body of work devoted to $\mu$-bases. The algorithm presented in this paper can be generalized to compute minimal bases of the kernels of $m\times n$ polynomial matrices. The details and comparison  with existing algorithms will be the subject of a  forthcoming work.

  \newpage
  \vskip5mm
  
\noindent\textbf{Acknowledgement:} We would like to thank Ron Goldman for a careful reading of the first version of our paper and Vincent Neiger for pointing out a body of literature devoted to the computation of minimal nullspaces of $m\times n$ polynomial matrices. We are grateful to the anonymous referees for a number of valuable suggestions and, in particular, for a suggestion to discuss the homogeneous version of the problem.  The research was partially supported by
the  US NSF CCF-1319632 grant. %%%%%
%biblio
%%%%%

\bibliographystyle{plain}
\bibliography{paper}

%\begin{thebibliography}{10}

%\end{thebibliography}

\end{document}